\documentclass{amsart}
\usepackage{amsmath,amssymb,amsthm,mathrsfs,amscd,ascmac}
\usepackage[all]{xy}
\usepackage[textwidth=125mm,textheight=185mm]{geometry}
\usepackage[dvipdfmx]{hyperref}
\usepackage{color}
\theoremstyle{plain}
\newtheorem{thm}{Theorem}[section]
\newtheorem{lemm}[thm]{Lemma}
\theoremstyle{definition}
\newtheorem{dfn}[thm]{Definition}
\theoremstyle{remark}
\newtheorem{rmk}[thm]{Remark}
\theoremstyle{Main Theorem}
\newtheorem*{mthm}{Main Theorem}
\theoremstyle{Conjecture}
\newtheorem{conj}[thm]{Conjecture}
\theoremstyle{Corollary}
\newtheorem{cor}[thm]{Corollary}
\begin{document}
\title{On\ the\ Feigin-Tipunin\ conjecture}
\author{SHOMA SUGIMOTO}
\address{Research Institute For Mathematical Sciences, Kyoto University, Kyoto 606-8502 JAPAN}
\email{shoma@kurims.kyoto-u.ac.jp}1
\keywords{Representation theory, Vertex operator algebra}
\subjclass[2010]{81T40}
\date{}
\maketitle
\begin{abstract}
We prove the Feigin-Tipunin conjecture \cite{FT} on the geometric construction of the logarithmic $W$-algebras $W(p)_Q$ associated with a simply-laced simple Lie algebra $\mathfrak{g}$ and $p\in\mathbb{Z}_{\geq2}$. 
\end{abstract}

\section{Introduction}
The triplet $W$-algebra \cite{FGST1}-\cite{FGST3}, \cite{AM1}-\cite{AM3}, \cite{NT}, \cite{TW} is one of the most well-known examples of $C_2$-cofinite and irrational vertex operator algebra \cite{FB,FHL,Kac}, and related to many interesting objects such as the tails of colored Jones polynomials and false theta functions \cite{BM1,CCFGH,MN}, quantum groups at roots of unity \cite{CGR,FGR, NT}, and the quantum geometric Langlands program \cite{CG,Cr1}. The triplet $W$-algebra is defined as the kernel of the narrow screening operator on the lattice vertex operator algebra associated with the rescaled root lattice of type $A_1$. The definition of the triplet $W$-algebra is naturally generalized to the case of an arbitrary simply-laced simple Lie algebra $\mathfrak{g}$ and $p\in\mathbb{Z}_{\geq 2}$, and we call these vertex operator algebras the {\it logarithmic $W$-algebras} and denote by $W(p)_Q$ (\cite{AM4,CG,Cr1,Cr2,FL,FT,M}). It has been expected and widely believed that the logarithmic $W$-algebras inherit interesting properties of the triplet $W$-algebra mentioned above (\cite{AM4}, \cite{BKM1}-\cite{CrM}, \cite{FT,GLO}, \cite{M}). However, very little is known about the properties and the representation theory of $W(p)_Q$ except for the case of type $A_1$, that is the case of the triplet $W$-algebra,
and for the case where  
$\mathfrak{g}=\mathfrak{sl}_3$ and $p=2$ (\cite{AMW}).



The main purpose of this paper is to prove the geometric construction of $W(p)_Q$ 
conjectured in \cite{FT} in full generality.
 
%
Let $\mathfrak{g}=\mathfrak{n}_-\oplus\mathfrak{h}\oplus\mathfrak{n}_+$ be a simply-laced simple Lie algebra of rank $l$, $\mathfrak{b}=\mathfrak{n}_-\oplus\mathfrak{h}$ the Borel subalgebra of $\mathfrak{g}$, $G$ and $B$ the corresponding algebraic groups, respectively. Let $Q$ and $P_+$ be the root lattice and the set of dominant integral weights of $\mathfrak{g}$, respectively. 
For a fixed integer $p\in\mathbb{Z}_{\geq2}$, let $V_{\sqrt{p}Q}$ be the lattice vertex operator algebra associated with $\sqrt{p}Q$ (see \eqref{V_A}), 
and for $\lambda\in\Lambda$ (see \eqref{(15)} for the definition), let $V_{\sqrt{p}Q+\lambda}$ be the irreducible $V_{\sqrt{p}Q}$-module. 
Then for each $\lambda\in\Lambda$, the action of {\it screening operators} $\{f_i\}_{i=1}^l$ 
extends to an integrable action of $\mathfrak{b}$ on $V_{\sqrt{p}Q+\lambda}$.
Thus,
we obtain the homogeneous vector bundle $$\xi_{\lambda}=G\times_B V_{\sqrt{p}Q+\lambda}$$ 
over the flag variety $G/B$.
The space of global sections $H^0(\xi_0)$ inherits a vertex operator algebra structure from $V_{\sqrt{p}Q}$, and $H^0(\xi_{\lambda})$ is an $H^0(\xi_0)$-module for every $\lambda\in\Lambda$. 
Also, there is an obvious action of $G$ on $H^0(\xi_\lambda)$ that gives $H^0(\xi_0)$-module automorphisms.

On the other hand, we have the intertwining operators $\{F_{i,\lambda}\}_{i=1}^l$ 
from $V_{\sqrt{p}Q+\lambda}$ to $V_{\sqrt{p}Q+\sigma_i\ast\lambda}$ (see \eqref{(16.2)}), which are called the {\it narrow screening operators}. 
The  logarithmic $W$-algebra $W(p)_Q$ is by definition  \cite{Cr1,CrM}
the intersection of the kernel of $F_{i,0}$ in $ V_{\sqrt{p}Q}$.
Similarly,
for each $\lambda\in\Lambda$,
the $W(p)_Q$-module $W(p,\lambda)_Q$  is defined as 
the  intersection of the kernel of narrow screening operators
in  $V_{\sqrt{p}Q+\lambda}$
(see Definition \ref{logWalgs} for the precise definition).

We also define the vertex subalgebra $\mathcal{W}_0$ of $W(p)_Q$ and the $\mathcal{W}_0$-submodule $\mathcal{W}_{\sqrt{p}\alpha+\lambda}$ of $W(p,\lambda)_Q$ for each $\alpha\in Q$ as in Definition \ref{dfnmathcalW}. 
The vertex algebra
$\mathcal{W}_0$ contains the principal 
$W$-algebra $\mathcal{W}^{p-h^{\vee}}(\mathfrak{g})$ (\cite{FF})
 at level $p-h^\vee$ as its vertex subalgebra, where $h^{\vee}$ is the dual Coxeter number of $\mathfrak{g}$.

The fiber of $\xi_{\lambda}$ at the identity $\operatorname{id}B\in G/B$
is
naturally identified with $V_{\sqrt{p}Q+\lambda}$.
In Section \ref{lastsection} below, we show that the map
\begin{align}
\Phi_\lambda\colon H^0(\xi_{\lambda})\rightarrow V_{\sqrt{p}Q+\lambda},
\quad s\mapsto s(\operatorname{id}B)
\label{eq:fiber}
\end{align}
is an embedding and its
 image is contained in $W(p,\lambda)_Q\subset V_{\sqrt{p}Q+\lambda}$.

\begin{mthm} $ $

\begin{enumerate}
\item\label{mthm1}
 The map \eqref{eq:fiber}
gives a vertex operator algebra isomorphism
\begin{align}
H^0(\xi_0)\simeq W(p)_Q.\nonumber
\end{align}
In particular, the group $G$ acts on $W(p)_Q$ as an automorphism group.
\item\label{mthm2}
 More generally, for $\lambda\in \Lambda$,
the embedding
$H^0(\xi_\lambda)\hookrightarrow W(p,\lambda)_Q$
is an isomorphism
if and only if  $\lambda$ satisfies the condition \eqref{novelcond3}.
\item\label{mthm3}
 For $\lambda\in \Lambda$,
 we have a $G\times\mathcal{W}_0$-module isomorphism
\begin{align*}
H^0(\xi_{\lambda})\simeq \bigoplus_{\alpha\in P_+\cap Q}\mathcal{R}_{\alpha+\hat\lambda}\otimes\mathcal{W}_{-\sqrt{p}\alpha+\lambda},
\end{align*}
where $R_{\mu}$ is the irreducible $G$-module with the highest weight $\mu$.
\end{enumerate}
\end{mthm}
In particular, $\mathcal{W}_0$ is the $G$-orbifold of $W(p)_Q$. 
\begin{rmk}
In the subsequent paper \cite{S1} we will prove 
the simplicity and $\mathcal{W}_{p-h^\vee}(\mathfrak{g})$-module structure of $H^0(\xi_\lambda)$
for any simply-laced simple $\mathfrak{g}$ and $\lambda\in\Lambda$ such that $(\sqrt{p}\bar\lambda+\rho,\theta)\leq p$.
\end{rmk}
We also confirm the following two assertions stated in \cite{FT} without details of
proofs.

\begin{thm}\label{thm1}
For $p\geq h^\vee-1$, 
we have
\begin{align*}
\mathcal{W}_0\simeq\mathcal{W}^{p-h^{\vee}}(\mathfrak{g})
\end{align*}
as vertex algebras.
\end{thm}

\begin{thm}\label{thm2}
If $\lambda\in\Lambda$ satisfies $(\sqrt{p}\bar\lambda+\rho,\theta)\leq p$,
then we have $H^n(\xi_{\lambda})=0$ for $n>0$ and
\begin{align}
\operatorname{tr}_{H^0(\xi_{\lambda})}(q^{L_0-\frac{c}{24}}z_1^{h_{1,\lambda}}\cdots z_l^{h_{l,\lambda}})=\sum_{\alpha\in P_+\cap Q}\chi_{\alpha+\hat{\lambda}}^{\mathfrak{g}}(z)\operatorname{tr}_{T^p_{\sqrt{p}\bar\lambda,\alpha+\hat\lambda}}(q^{L_0-\frac{c}{24}}),\nonumber
\end{align}
where $\chi_{\alpha+\hat{\lambda}}^{\mathfrak{g}}(z)$ is the Weyl character formula of $\mathcal{R}_{\alpha+\hat\lambda}$, and $T^p_{\sqrt{p}\bar\lambda,\alpha+\hat\lambda}$ is the $\mathcal{W}_{p-h^\vee}(\mathfrak{g})$-module defined in \cite[(2.10)]{ArF}.
\end{thm}
Note that there is an explicit  formula
of
$\operatorname{tr}_{T^p_{\sqrt{p}\bar\lambda,\alpha+\hat\lambda}}(q^{L_0-\frac{c}{24}})$,
see  \eqref{(61)}.

\begin{rmk}
In \cite{FT}, it is 
stated
that Theorem \ref{thm1}, \ref{thm2} hold for any $\lambda\in\Lambda$ and $p\in\mathbb{Z}_{\geq2}$. However, because of a certain technical difficulty (see Remark \ref{tech} in Section \ref{vanishingcohoms}), we prove the 
assertions only under the assumption that
$(\sqrt{p}\bar\lambda+\rho,\theta)\leq p$, that is, $\sqrt{p}\bar\lambda$ is in the closure of the fundamental alcove \cite{J}.
\end{rmk}

\begin{rmk}
By Lemma \ref{condequiv}, if $\lambda\in\Lambda$ satisfies $(\sqrt{p}\bar\lambda+\rho,\theta)\leq p$, then $\lambda$ satisfies \eqref{novelcond3}. However, the converse is not true in general.
\end{rmk}

The paper is organized as follows. In Section \ref{sectionpre1} and \ref{sectionpre2}, we shall refer to notation and results given in \cite{AM1,FT,NT}. Section \ref{sectionproofs} is devoted to the proofs of the main results. In Section \ref{parabolic}, for any $\lambda\in\Lambda$ and $p\in\mathbb{Z}_{\geq2}$, we show that $W(p,\lambda)_{Q}^i$ (see Definition \ref{logWalgs}) is a $P_i$-submodule of $V_{\sqrt{p}Q+\lambda}$ in the sense of Definition \ref{dfnGsubmodule}. Applying this, in Section \ref{vanishingcohoms} and \ref{chformulas}, we prove Theorem \ref{thm2} along the line of the proof in \cite{FT}.

A key ingredient of our proof of Main Theorem is the construction of the largest $G$-submodule of $V_{\sqrt{p}Q+\lambda}$ in the sense of Definition \ref{dfnGsubmodule} in two different ways. 
If $V$ is a $G$-submodule of $V_{\sqrt{p}Q+\lambda}$, then we have
\begin{align}
V=\Phi_\lambda(H^0(G\times_BV))\subseteq\Phi_\lambda(H^0(\xi_\lambda)).
\end{align}
Therefore, all $G$-submodules of $V_{\sqrt{p}Q+\lambda}$ are $G$-submodules of $\Phi_\lambda(H^0(\xi_\lambda))$.
On the other hand, in Section \ref{vanishingcohoms}, we give a $P_i$-module isomorphism
\begin{equation}
H^0(P_i/B,P_i\times_BV_{\sqrt{p}Q+\lambda})\simeq W(p,\lambda)_Q^i,~s\mapsto s(\operatorname{id}B).\nonumber
\end{equation}
For the same reason as above, all $P_i$-submodules of $V_{\sqrt{p}Q+\lambda}$ are $P_i$-submodules of $W(p,\lambda)_Q^i$. 
In particular, $W(p,\lambda)_Q$ contains $\Phi_\lambda(H^0(\xi_\lambda))$. By Theorem \ref{parab} in Section \ref{almostproof}, $W(p,\lambda)_Q$ is closed under the $G$-action if and only if $\lambda$ satisfies the condition \eqref{novelcond3}. 
Finally, after we prove Main Theorem \eqref{mthm3}, we prove Theorem \ref{thm1} by combining Main Theorem \eqref{mthm3} with the character formulas in Section \ref{chformulas}.\\

{\it Acknowledgements} \  \ This paper is the master thesis of the author, and he wishes to express his gratitude to his supervisor Tomoyuki Arakawa for suggesting the problems and lots of advice to improve this paper. He also gives an address to thanks his previous supervisor Hiraku Nakajima for two years' detailed guidance. He is deeply grateful to Naoki Genra and Ryo Fujita for their many pieces of advice and encouragement. He thanks Thomas Creutzig, Boris Feigin, and Shigenori Nakatsuka for useful comments and discussions. In particular, \cite{FT} is the original idea of this paper, and discussions with Thomas Creutzig and Boris Feigin were very suggestive and interesting for future works of the paper related to their recent works. 
Finally, he appreciates the referee for the thoughtful and constructive feedback.
This work was supported by JSPS KAKENHI Grant Number 19J21384.

\section{Preliminaries I}\label{sectionpre1}
\subsection{Notation}\label{subsect:notation}
We fix an integer $p$ such that $p\geq 2$ throughout the present paper.
Let $\mathfrak{g}$ be a simply-laced simple Lie algebra of rank $l$, and $\mathfrak{g}=\mathfrak{n}_-\oplus\mathfrak{h}\oplus\mathfrak{n}_+$ the triangular decomposition, $\mathfrak{h}$ the Cartan subalgebra, $\mathfrak{b}$ the Borel subalgebra $\mathfrak{n}_-\oplus\mathfrak{h}$, $G$, $H$, and $B$ the semisimple, simply-connected, complex algebraic groups corresponding to $\mathfrak{g}$, $\mathfrak{h}$, $\mathfrak{b}$, respectively. 
We adopt the standard numbering for the simple roots $\alpha_1,\ldots,\alpha_l$ of $\mathfrak{g}$ as in \cite{B} and $\omega_1,\dots,\omega_l$ denote the corresponding fundamental weights.
 Let $Q$ be the root lattice, $P$ the weight lattice, $P_+$ the set of dominant integral weights of $\mathfrak{g}$, respectively.
  Denote by $(\cdot , \cdot)$ the invariant form of $\mathfrak{g}$ normalized as $(\alpha_i,\alpha_i)=2$ for any $1\leq i\leq l$, $W$ the Weyl group of $\mathfrak{g}$ generated by the simple reflections $\{\sigma_i\}_{i=1}^l$, $(c_{ij})$ the Cartan matrix of $\mathfrak{g}$ and $(c^{ij})$ the inverse matrix to $(c_{ij})$, $\rho$ the half sum of positive roots, $\theta$ the highest root, $h$ the Coxeter number of $\mathfrak{g}$. 
 We choose the set $\hat\Lambda$ of representative of generators of the abelian group $P/Q$ in $P_+$ as \cite{Kac2}, that is, $\hat\lambda\in\hat\Lambda$ satisfies $(\hat\lambda,\theta)=1$). Let $h_1,\cdots,h_l$ be the basis of $\mathfrak{h}$ corresponding to the simple roots by $(\cdot,\cdot)$. 
For $\sigma\in W$, $l(\sigma)$ denotes the length of $\sigma$.
For $\alpha\in P_+$, we write $\mathcal{R}_\alpha$ for the irreducible $\mathfrak{g}$-module with the highest weight $\alpha$.

Denote by $\Pi$ the set of nodes of the Dynkin diagram of $\mathfrak{g}$. We identify $\Pi$ with the set of integers $\{1,\dots,l\}$. For a subset $J$ of $\Pi$, let $P_J$ denote the parabolic subgroup of $G$ corresponding to $J$.
 When $J$ is a one-point subset $\{j\}$, we use the letter $P_j$ instead of $P_{\{j\}}$.
 For a $B$-module $V$, $P_J\times V$ has a $B$-module structure by
 \begin{align}
 b(g,v):=(gb^{-1}, bv)
 \end{align}
 for $g\in P_J$, $b\in B$, and $v\in V$.
 Then we have the homogeneous vector bundle $P_J\times_BV$ over $P_J/B$ defined by
 \begin{align}
 (P_J\times V)/B\rightarrow P_J/B,~[gB,v]\mapsto gB.
 \end{align}
Here, the $P_J$-action on $P_J/B$ is given by right multiplication, and that on $(P_J\times V)/B$ is given by $g[g_0,v]:=[gg_0,v]$ for $g,g_0\in P_J$ and $v\in V$.

\hypertarget{$V(\mu)$}{For $\mu\in\mathfrak{h}^*$, let $\mathbb{C}_{\mu}$ be the one-dimensional $B$-module such that for each $1\leq i\leq l$, $h_{i}$ acts on $\mathbb{C}_{\mu}$ by $(\alpha_i , \mu)\operatorname{id}$ and $\mathfrak{n}_{-}$ acts on $\mathbb{C}_{\mu}$ trivially. For a $B$-module $V$ and $\mu\in\mathfrak{h}^\ast$, we write $V(\mu)$ for the $B$-module $V\otimes\mathbb{C}_{\mu}$.}

\hypertarget{$O(\mu)$}{We use the same notation for a vector bundle and its sheaf of sections.
For a sheaf $\mathcal{F}$ over a topological space $X$ and an open subset $U$ of $X$, $\mathcal{F}(U)$ denotes the space of sections of $\mathcal{F}$ on $U$.
For an algebraic variety $X$, let $\mathcal{O}_X$ be the structure sheaf of $X$.
For $\mu\in\mathfrak{h}^\ast$, we write $\mathcal{O}(\mu)$ for the line bundle (or invertible sheaf) $G\times_B\mathbb{C}_\mu$ over the flag variety $G/B$. In particular, $\mathcal{O}(0)$ is $\mathcal{O}_{G/B}$. For an $\mathcal{O}_{G/B}$-module $\mathcal{F}$, we use the letter $\mathcal{F}(\mu)$ for $\mathcal{F}\otimes_{\mathcal{O}_{G/B}}\mathcal{O}(\mu)$. In particular, for a $B$-module $V$ and the homogeneous vector bundle $G\times_BV$ over $G/B$, we have $(G\times_BV)(\mu)=G\times_BV(\mu)$.}

For a vertex operator algebra $V$ and $v\in V$, the field of $v\in V$ is denoted by
\begin{equation}\label{00}
v(z)=\sum_{n\in\mathbb{Z}}v_{(n)}z^{-n-1}.
\end{equation}
We call $v_{(0)}$ the zero mode of $v(z)$.
For a conformal vector $\omega\in V$, we write $L_n$ for the Virasoro operator $\omega_{(n+1)}$.
For a vertex operator algebra $V$ and a Lie algebra $\mathfrak{a}$, denote by $\mathcal{U}(V)$ and $\mathcal{U}(\mathfrak{a})$ the universal enveloping algebras of $V$ (see \cite{FB,Kac}) and $\mathfrak{a}$, respectively. For a vertex operator algebra $V$, a $V$-module $M$ and $m\in M$, we write $\mathcal{U}(V)m$ for the $V$-submodule of $M$ generated by $m$. \hypertarget{confdimension}{For $\Delta\in\mathbb{C}$, denote by $M_{\Delta}$ the space of homogeneous vectors with the conformal weight $\Delta$.}
Finally, when a $V$-module $M$ has a $B$-module structure, for $\mu\in P$, we identify $M$ with $M(\mu)=M\otimes\mathbb{C}_{\mu}$ as $V$-modules by the linear isomorphism
$M\xrightarrow{\sim}M(\mu),~v\mapsto v\otimes1$.

\subsection{$P_J$-submodule of a $B$-module}\label{subsect:P_Jsubmod}
Let $J$ be a subset of $\Pi$ and $P_J$ be the corresponding parabolic subgroup of $G$.
Let $\pi\colon E\rightarrow X$ be a $P_J$-equivariant vector bundle.
Then the space of global sections (in other words, the zeroth sheaf cohomology) $H^0(\pi)$ of $\pi\colon E\rightarrow X$ has the $P_J$-module structure as follows:
For $g\in P_J$ and $s\in H^0(\pi)$, let
\begin{align}\label{gactiononglobalsection}
gs=g\circ s\circ g^{-1},
\end{align}
where the action $g^{-1}$ in \eqref{gactiononglobalsection} is that on $X$, and the action $g$ in \eqref{gactiononglobalsection} is that on $E$. Since $\pi\colon E\rightarrow X$ is $P_J$-equivariant, \eqref{gactiononglobalsection} is well-defined.
Note that for a $B$-module $V$, the homogeneous vector bundle $P_J\times_BV$ is $P_J$-equivariant.

Let $V$ be a $B$-module. 
we define the $B$-action on $\mathcal{O}_{P_J}(P_J)\otimes_{\mathbb{C}}V$ by 
\begin{align}\label{(31.15)}
(bs)(g)=b(s(gb))
\end{align}
for $s\in\mathcal{O}_{P_J}(P_J)\otimes_{\mathbb{C}}V$, $g\in P_J$, $b\in B$.
Denote by $(\mathcal{O}_{P_J}(P_J)\otimes_{\mathbb{C}}V)^B$ the orbifold of the $B$-action \eqref{(31.15)} on $\mathcal{O}_{P_J}(P_J)\otimes_{\mathbb{C}}V$.
Then $(\mathcal{O}_{P_J}(P_J)\otimes_{\mathbb{C}}V)^B$ has the $P_J$-module structure as follows: 
For $g,g_0\in P_J$ and $f\in(\mathcal{O}_{P_J}(P_J)\otimes_{\mathbb{C}}V)^B$, let
\begin{align}\label{(31.2)}
(gf)(g_0)=f(g^{-1}g_0).
\end{align}

\begin{lemm}\label{lemm:P_Jmodisom}
Let $V$ be a $B$-module and $J\subseteq \Pi$.
Then we have the $P_J$-module isomorphism
\begin{align}\label{(31.1)}
(\mathcal{O}_{P_J}(P_J)\otimes_{\mathbb{C}}V)^B\simeq H^0(P_J\times_BV),~f\mapsto s_f,
\end{align} 
where $s_f$ is defined by $s_f(gB)=[g,f(g)]$.
\end{lemm}
\begin{proof}
Since $f\in(\mathcal{O}_{P_J}(P_J)\otimes_{\mathbb{C}}V)^B$, $s_f$ is well-defined.
By definition, \eqref{(31.1)} is injective.
For $s\in H^0(P_J\times_BV)$ and $g\in P_J$, let $s(gB)=[g',v]\in(P_J\times V)/B$. Then there exists a unique element $b\in B$ such that $g'=gb$. Let us define $f_s$ in $\mathcal{O}_{P_J}(P_J)\otimes_{\mathbb{C}}V$ by $f_s(g)=bv$. Then $f_s$ is independent of the choice of $[g',v]$, and hence $f_s\in(\mathcal{O}_{P_J}(P_J)\otimes_{\mathbb{C}}V)^B$. 
By definition, we have $s_{f_s}=s$, and hence \eqref{(31.1)} is surjective.
By \eqref{gactiononglobalsection} and \eqref{(31.2)}, the bijection \eqref{(31.1)} gives a $P_J$-module isomorphism.
\end{proof}

Keeping the isomorphism \eqref{(31.1)} in mind, we shall sometimes abuse notation and denote 
the inverse image $f_s\in(\mathcal{O}_{P_J}(P_J)\otimes_{\mathbb{C}}V)^B$ of $s\in H^0(P_J\times_BV)$ by $s$.

\begin{lemm}\label{lemm:H^0VandV}
Let $J$ be a subset of $\Pi$ and $M$ be a $P_J$-module.
Then we have the $P_J$-module isomorphism 
\begin{align}\label{tuika20}
&H^0(P_J\times_BM)\xrightarrow{\sim}M,
~s\mapsto s(\operatorname{id}).
\end{align}
\end{lemm}
\begin{proof}
Let $P_J/B\times M$ be the $P_J$-equivariant trivial vector bundle over $P_J/B$ such that the $P_J$-action on $P_J/B$ is right multiplication, and that on $P_J/B\times M$ is given by $g(g_0B,v):=(gg_0B,gv)$.
Then the following map gives the $P_J$-equivariant vector bundle isomorphism:
\begin{align}\label{isomtrivbdle}
 \iota\colon P_J\times_BM\rightarrow P_J/B\times M,\ [g_0,v]\mapsto (g_0B,g_0v).
\end{align}
In particular, we have the $P_J$-module isomorphism
\begin{align}\label{2/20,1}
H^0(P_J\times_BM)\xrightarrow{\sim}H^0(P_J/B\times M),~s\mapsto\iota\circ s.
\end{align}
On the other hand, 
by combining the Levi decomposition of $P_J$ with the Borel-Weil theorem \cite{Kum,Kum1},
 we have 
 $H^0(\mathcal{O}_{P_J/B})\simeq\mathbb{C}$ as $P_J$-modules.
 Thus, we have the $P_J$-module isomorphism
\begin{align}\label{2/20,2}
H^0(P_J/B\times M)\simeq H^0(\mathcal{O}_{P_J/B})\otimes M\simeq M,~\iota\circ s\mapsto p_2 ((\iota\circ s)(\operatorname{id}B)),
\end{align}
where $p_2$ is the projection $P_J/B\times M\rightarrow M$.
Because $p_2(\iota\circ s(\operatorname{id}B))=s(\operatorname{id})$, combining \eqref{2/20,1} with \eqref{2/20,2}, we obtain \eqref{tuika20}.
\end{proof}

  \begin{cor}\label{cor:P_Jsubmod}
 Let $J$ be a subset of $\Pi$, $M_1$ and $M_2$ be $P_J$-modules.
 If $\phi\colon M_1\simeq M_2$ is a $B$-module isomorphism, then $\phi$ is a $P_J$-module isomorphism.
 \end{cor}
 \begin{proof}
By Lemma \ref{lemm:H^0VandV}, we obtain the $P_J$-module isomorphisms
\begin{align}
&M_1\xrightarrow{\sim}H^0(P_J\times_BM_1)\xrightarrow{\sim}H^0(P_J\times_BM_2)\xrightarrow{\sim}M_2,\\
&v\mapsto s_v\mapsto \phi\circ s_v=s_{\phi(v)}\mapsto \phi(v),\nonumber
\end{align}
where $s_v\in H^0(P_J\times_BM_1)$ and $s_{\phi(v)}\in H^0(P_J\times_BM_2)$ are defined by $s_v(\operatorname{id})=v$ and $s_{\phi(v)}(\operatorname{id})=\phi(v)$, respectively.
\end{proof}

\hypertarget{$G$-submodule}{
 \begin{dfn}\label{dfnGsubmodule}
Let $J$ be a subset of $\Pi$. For a $B$-module $V$ and a $B$-submodule $M$ of $V$, we call $M$ a {\it $P_J$-submodule} of $V$ if for some $P_J$-module $M'$, we have $M'\simeq M$ as $B$-modules.
Note that by Corollary \ref{cor:P_Jsubmod}, a $P_J$-submodule structure on $M$ is unique if it exists.
 \end{dfn}}

 In Definition \ref{dfnGsubmodule}, if $M$ is a $P_J$-sumbodule of $V$, then we have the $B$-module isomorphism
 \begin{align}\label{2/20,3}
 H^0(P_J\times_BM)\xrightarrow{\sim}M, ~s\mapsto s(\operatorname{id}).
 \end{align}
 In fact, for a $P_J$-module $M'$ and a $B$-module isomorphism $\phi\colon M'\simeq M$, 
 we have the $B$-module isomorphisms
 \begin{align}
 &H^0(P_J\times_BM)\xrightarrow{\sim} H^0(P_J\times_BM')\xrightarrow{\sim}M'\xrightarrow{\sim} M,\\
 &s\mapsto \phi^{-1}\circ s\mapsto \phi^{-1}(s(\operatorname{id}))\mapsto s(\operatorname{id}),\nonumber
 \end{align}
 where the second isomorphism follows from Lemma \ref{lemm:H^0VandV}.

\begin{lemm}\label{lemm:maximalP_Jsubmod}
Let $J$ be a finite subset of $\Pi$. For a $B$-module $V$ and a $P_J$-submodule $M$ of $V$, 
 if the $B$-module homomorphism
\begin{align}
\Phi\colon H^0(P_J\times_BV)\rightarrow V,~s\mapsto s(\operatorname{id})
\end{align}
is an embedding and its image is $M\subseteq V$,
then $M$ is the largest $P_J$-submodule of $V$.
\end{lemm}
\begin{proof}
For any $B$-submodule $M_0$ of $V$, $H^0(P_J\times_BM_0)$ is a $P_J$-submodule of $H^0(P_J\times_BV)$.
By \eqref{2/20,3}, if $M_0$ is a $P_J$-submodule of $V$, then we have
\begin{align}
M_0=\Phi(H^0(P_J\times_BM_0))\subseteq \Phi(H^0(P_J\times_BV))=M.
\end{align}
Thus, $M_0$ is a $P_J$-submodule of $M$.
\end{proof}

\subsection{Lattice vertex operator algebra $V_{\sqrt{p}Q}$}\label{subsect:lattice}
We consider the lattice vertex operator algebra $V_{\sqrt{p}Q}$ associated with the rescaled root lattice $\sqrt{p}Q$ (\cite{FLM},\cite[Section 5.4]{Kac}). We reproduce the definition for completeness. Let
\begin{equation}
\hat{\mathfrak{h}}=\mathfrak{h}\otimes\mathbb{C}[t^{\pm}]\oplus\mathbb{C}K
\end{equation}
be the Heisenberg algebra associated with $\mathfrak{h}$ and the invariant form $(\cdot , \cdot)$ in Section \ref{subsect:notation}, and let
$\hat{\mathfrak{h}}^{<0}$ and $\hat{\mathfrak{h}}^{\geq0}$ be the commutative Lie subalgebras of $\hat{\mathfrak{h}}$ defined by
\begin{align}
\hat{\mathfrak{h}}^{<0}=\sum_{j<0}\mathfrak{h}\otimes t^{j},~\hat{\mathfrak{h}}^{\geq0}=\sum_{j\geq0}\mathfrak{h}\otimes t^{j},
\end{align}
respectively.
For a subset $A$ of $\frac{1}{\sqrt{p}}P$, denote by $V_A$ the vector space defined by
\begin{align}\label{V_A}
V_A=\bigoplus_{\nu\in A}\mathcal{U}(\hat{\mathfrak{h}})\otimes_{\mathcal{U}(\hat{\mathfrak{h}}^{\geq0})}\mathbb{C}|\nu\rangle=\bigoplus_{\nu\in A}\mathcal{F}_{\nu},
\end{align}
where $|\nu\rangle$ is the vector corresponding to the lattice point $\nu\in\frac{1}{\sqrt{p}}P$, 
$\mathcal{F}_0$ is the Heisenberg vertex operator algebra defined by $\mathcal{F}_0=\mathcal{U}(\hat{\mathfrak{h}}^{<0})\otimes\mathbb{C}|0\rangle$, and 
$\mathcal{F}_{\nu}$ is the Fock representation $\mathcal{U}(\hat{\mathfrak{h}})\otimes_{\mathcal{U}(\hat{\mathfrak{h}}^{\geq0})}\mathbb{C}|\nu\rangle$ of $\mathcal{F}_0$ with the highest weight $\nu$.

 Let us set up some notations to describe the vertex algebra structure on $V_{\sqrt{p}Q}$.
 For $a\in\mathfrak{h}$ and $\mu\in\frac{1}{\sqrt{p}}P$, the fields $a(z)$ and $V_{|\mu\rangle}(z)$ 
 on $V_{\frac{1}{\sqrt{p}}P}$ are given by
 \begin{align}
 &a(z)=\sum_{n\in\mathbb{Z}}a_{(n)}z^{-n-1}=\sum_{n\in\mathbb{Z}}(a\otimes t^n)z^{-n-1},\\
&V_{|\mu\rangle}(z)=e^{\mu}z^{\mu_{(0)}}\operatorname{exp}{(-\sum_{n<0}\frac{z^{-n}}{n}\mu_{(n)})}\exp{(-\sum_{n>0}\frac{z^{-n}}{n}\mu_{(n)})},
 \end{align}
 respectively.
 Here the operator $e^{\mu}\in\operatorname{End}_{\mathbb{C}}(V_{\frac{1}{\sqrt{p}}P})$ is defined by
 \begin{align}
  e^{\mu}|\nu\rangle=|\mu+\nu\rangle,~
  [a_{(n)}, e^{\mu}]=\delta_{n,0}(a,\mu)e^{\mu}
 \end{align}
 for $a\in\mathfrak{h}$ and $\nu\in\frac{1}{\sqrt{p}}P$.
 On the other hand, by \cite[Corollary 5.5]{Kac}, there exists a $2$-cocycle $\epsilon\colon\sqrt{p}Q\times\sqrt{p}Q\rightarrow\mathbb{C}^{\times}$ that satisfies \cite[(5.4.14)]{Kac}.
 Moreover, by \cite[Section 5.4.5]{FB}, for $\mu,\nu\in\frac{1}{\sqrt{p}}P$ such that $(\mu,\nu)\in\mathbb{Z}$, $\epsilon(\mu,\nu)\in\mathbb{C}^\times$ is also well-defined.
 Let us fix such a $2$-cocycle $\epsilon$ throughout the present paper.
 For $\mu\in\frac{1}{\sqrt{p}}P$, set
$(\frac{1}{\sqrt{p}}P)_{\mu}=\{\nu\in\frac{1}{\sqrt{p}}P~|~(\mu,\nu)\in\mathbb{Z}\}$.
Denote by $|\mu\rangle(z)$ the formal power series defined by
\begin{align}
|\mu\rangle(z)=V_{|\mu\rangle}(z)\epsilon(\mu,\cdot)\in\operatorname{Hom}_{\mathbb{C}}(V_{(\frac{1}{\sqrt{p}}P)_{\mu}},V_{(\frac{1}{\sqrt{p}}P)_{\mu}+\mu})\otimes_{\mathbb{C}}\mathbb{C}[[z^{\pm}]],
\end{align}
where the operator $\epsilon(\mu,\cdot)$ is defined by $\epsilon(\mu,\cdot)|_{\mathcal{F}_\nu}=\epsilon(\mu,\nu)\operatorname{id}_{\mathcal{F}_\nu}$ for each $\nu\in(\frac{1}{\sqrt{p}}P)_\mu$.
 Then the vertex algebra structure on $V_{\sqrt{p}Q}$ is defined by
 \begin{align}\label{V_pQfields}
 {V_{\sqrt{p}Q}}(a,z)=a(z),~{V_{\sqrt{p}Q}}(|\sqrt{p}\alpha\rangle,z)=|\sqrt{p}\alpha\rangle(z).
 \end{align}
 By \cite[Theorem 5.5]{Kac}, the vertex algebra structure on $V_{\sqrt{p}Q}$ is unique and independent of the choice of $\epsilon$.
By \cite[(5.5.13), (5.5.14), (5.5.18)]{Kac}, for $a,b\in\mathfrak{h}$ and $\mu,\nu\in \sqrt{p}Q$, we have the following OPE formulas:
 \begin{align}
a(z)b(w)\sim&\frac{(a,b)}{(z-w)^2},\label{(11.1)}\\
a(z)|\mu\rangle(w)\sim&\frac{(\mu,a)|\mu\rangle(w)}{z-w},\label{(11.2)}\\
 |\mu\rangle(z)|\nu\rangle(w)\sim&\epsilon(\mu,\nu)(z-w)^{(\mu,\nu)}\sum_{n\in\mathbb{Z}_{\geq0}}\sum_{\substack{k_1+2k_2+\cdots+nk_n=n\\ k_i\in\mathbb{Z}_{\geq0}}}\label{(11.3)}\\
 &\frac{(z-w)^n:\mu(w)^{k_1}(\partial\mu(w))^{k_2}\cdots(\partial^{n-1}\mu(w))^{k_n}|\mu+\nu\rangle(w):}{(1!)^{k_1}k_1!(2!)^{k_2}k_2!\cdots(n!)^{k_n}k_n!}\nonumber.
 \end{align}

 We choose the shifted conformal vector $\omega$ of $V_{\sqrt{p}Q}$ as
 \begin{align}\label{(11.5)}
 \omega=\frac{1}{2}\sum_{1\leq i,j \leq l}c^{ij}(\alpha_{i})_{(-1)}\alpha_{j}+Q_{0}(\rho)_{(-2)}|0\rangle\in\mathcal{F}_0\subseteq V_{\sqrt{p}Q},
 \end{align}
 where $Q_{0}=\sqrt{p}-\frac{1}{\sqrt{p}}$.
 The central charge $c$ of $\omega$ is given by
 \begin{equation}\label{(13)}
 c=l+12(\rho, \rho)(2-p-\frac{1}{p})=l+h\dim{\mathfrak{g}}(2-p-\frac{1}{p}).
 \end{equation}

 Let $\bar{\Lambda}_p$ be a subset of $\frac{1}{\sqrt{p}}P$ defined by
 \begin{align}
 \bar{\Lambda}_p=\{\sum_{i=1}^l\frac{s_i}{\sqrt{p}}\omega_i~|~0\leq s_i\leq p-1\}.
 \end{align}
 For $\lambda\in\frac{1}{\sqrt{p}}P$, denote by $\bar\lambda$ the representative of $[\lambda]\in\frac{1}{\sqrt{p}}P/\sqrt{p}P$ in $\bar\Lambda_p$.
 Also, for $\lambda\in\frac{1}{\sqrt{p}}P$, we write $\hat\lambda$ for the representative of $[\frac{1}{\sqrt{p}}(\lambda-\bar\lambda)]\in P/Q$ in $\hat\Lambda$.
In particular, for any $\lambda\in\frac{1}{\sqrt{p}}P$, we have $\lambda=-\sqrt{p}(\alpha+\hat\lambda)+\bar\lambda$ for some $\alpha\in Q$.
Then the parameter set $\Lambda$ of irreducible $V_{\sqrt{p}Q}$-modules is given by
\begin{align}\label{(15)}
\Lambda=\{\lambda=-\sqrt{p}\hat{\lambda}+\bar{\lambda}~|~\hat\lambda\in\hat\Lambda,~ \bar\lambda\in\bar\Lambda_p\}
\end{align}
(see \cite{D}). For $\lambda\in\Lambda$, let $V_{\sqrt{p}Q+\lambda}$ be the irreducible $V_{\sqrt{p}Q}$-module defined by
\begin{align}\label{(15.5)}
V_{\sqrt{p}Q+\lambda}=\bigoplus_{\alpha\in Q}\mathcal{F}_{-\sqrt{p}\alpha+\lambda}=\bigoplus_{\alpha\in Q}\mathcal{F}_{-\sqrt{p}(\alpha+\hat\lambda)+\bar\lambda}.
 \end{align}
 Note that the $V_{\sqrt{p}Q}$-module structure of $V_{\sqrt{p}Q+\lambda}$ is defined by the fields
 \eqref{V_pQfields}.
 
 For $\mu\in\frac{1}{\sqrt{p}}P$, the conformal weight $\Delta_{\mu}$ of $|\mu\rangle$ is given by
 \begin{equation}\label{(14)}
 \Delta_{\mu}=\frac{1}{2}|\mu -Q_{0}\rho|^2 +\frac{c-l}{24}=\frac{1}{2}|\mu|^2 -Q_{0}(\mu, \rho).
 \end{equation}
 \begin{rmk}\label{confwtrmk}
 Note that for any $\alpha\in Q$ and $\lambda\in\Lambda$, the conformal weight of a nonzero homogeneous vector in $\mathcal{F}_{-\sqrt{p}\alpha+\lambda}$ is greater than or equal to $\Delta_{-\sqrt{p}\alpha+\lambda}$.
 \end{rmk}
\begin{rmk}
To regard $|-\sqrt{p}\alpha+\lambda\rangle$ as the highest weight vector of $\mathcal{R}_{\alpha+\hat{\lambda}}$ in later sections, we changed the definition of $\Lambda$ from \cite{FT}, \cite{BM}, \cite{CrM}. For the same reason, we also changed the definition of the operators $h_{1,\lambda},\dots,h_{l,\lambda}$ that give rise to the action of $\mathfrak{h}$ on $V_{\sqrt{p}Q+\lambda}$ in Section \ref{sectionpre2}. In consequence of these changes, our Borel subalgebra $\mathfrak{b}$ is not $\mathfrak{n}_+\oplus\mathfrak{h}$ but $\mathfrak{n}_-\oplus\mathfrak{h}$, and our character formulas of the modules $H^0(\xi_{\lambda})$ over the logarithmic $W$-algebra $W(p)_Q$ in later sections are slightly different from \cite{BM}, \cite{CrM}.
\end{rmk}

\section{Preliminaries I\hspace{-.1 em}I}\label{sectionpre2}
\subsection{Screening operators and the $B$-action}\label{screenings}
For $1\leq i\leq l$, $\alpha\in Q$ and $\lambda\in\Lambda$, we consider the {\it screening operator} $f_i$ defined by
\begin{equation}\label{(16)}
f_i =|\sqrt{p}\alpha_{i}\rangle_{(0)}\in\operatorname{Hom}_{\mathbb{C}}(\mathcal{F}_{-\sqrt{p}{\alpha}+\lambda},\mathcal{F}_{-\sqrt{p}(\alpha-\alpha_i)+\lambda}).
\end{equation}
For $1\leq i \leq l$ and $\mu\in\mathfrak{h}^\ast$, we consider the operator $h_{i,\lambda}(\mu)$ defined by 
\begin{equation}\label{(30)}
h_{i,\lambda}(\mu)=-\frac{1}{\sqrt{p}}(\alpha_i )_{(0)}+\frac{1}{\sqrt{p}}(\alpha_i , \bar{\lambda}+\sqrt{p}\mu)\operatorname{id}\in\operatorname{End}_{\mathbb{C}}(V_{\sqrt{p}Q+\lambda}).
\end{equation}
When $\mu=0$, we use the letter $h_{i,\lambda}$ instead of $h_{i,\lambda}(\mu)$.
For $\mu\in\mathfrak{h}^\ast$, $1\leq i\leq l$, and $n\in\mathbb{Z}$, it is straightforward to verify that
\begin{align}\label{(19.0)}
[f_i , L_n]=[h_{i,\lambda}(\mu),L_n]=0.
\end{align}
In particular, $f_i$ and $h_{i,\lambda}(\mu)$ preserve the conformal grading.
Moreover, since $f_i$ and $(\alpha_i)_{(0)}$ are zero modes, for any $a\in V_{\sqrt{p}Q}$, $v\in V_{\sqrt{p}Q+\lambda}$, and $n\in\mathbb{Z}$, we have
\begin{align}
f_i(a_{(n)}v)&=(f_ia)_{(n)}v+a_{(n)}f_iv,\label{infinitesimalisom}\\
h_{i,\lambda}(\mu)(a_{(n)}v)&=(h_{i,0}a)_{(n)}v+a_{(n)}h_{i,\lambda}(\mu)v.\label{infinitesimalisom2}
\end{align}
Let us recall notation in \hyperlink{$V(\mu)$}{Section \ref{subsect:notation}}.
\begin{thm}[{\cite[Theorem 4.1]{FT}}]\label{ft41}$ $
The operators $\{f_i,h_{i,\lambda}\}_{i=1}^l$ give rise to an integrable action of $\mathfrak{b}$ on $V_{\sqrt{p}Q+\lambda}$.   
More generally, the operators $\{f_i,h_{i,\lambda}(\mu)\}_{i=1}^l$ give rise to the integrable action of $\mathfrak{b}$ on $V_{\sqrt{p}Q+\lambda}(\mu)$. 
%
\end{thm}

\begin{proof}
We include the proof for completeness.
By \eqref{(11.2)} and \eqref{(11.3)}, we have
\begin{equation}\label{(22)}
[h_{i,\lambda}(\mu), f_j]=-c_{ij}f_j
\end{equation}
for $1\leq i,j\leq l$
and 
\begin{equation}\label{(23)}
\operatorname{ad}(f_i)^{1-c_{ij}}|\sqrt{p}\alpha_j\rangle\in\mathcal{F}_{\sqrt{p}((1-c_{ij})\alpha_i+\alpha_j)}
\end{equation}
for $1\leq i\not=j\leq l$.
By \eqref{(14)} and \eqref{(19.0)}, the conformal weight of $\operatorname{ad}(f_i)^{1-c_{ij}}|\sqrt{p}\alpha_j\rangle$ is $1$. However, since 
$\Delta_{\sqrt{p}((1-c_{ij})\alpha_i+\alpha_j)}=2-c_{ij}>1$,
by Remark \ref{confwtrmk}, we have $\operatorname{ad}(f_i)^{1-c_{ij}}|\sqrt{p}\alpha_j\rangle=0$. 
Hence the operators $\{f_i,h_{i,\lambda}\}_{i=1}^l$ give rise to an action of $\mathcal{U}(\mathfrak{b})$ on $V_{\sqrt{p}Q+\lambda}$.
Also, for $\mu\in\mathfrak{h}^\ast$, the operators $\{f_i,h_{i,\lambda}(\mu)\}_{i=1}^l$ give rise to the action of $\mathcal{U}(\mathfrak{b})$ on $V_{\sqrt{p}Q+\lambda}(\mu)$ by definition.

We show that the action of $\mathcal{U}(\mathfrak{b})$ on $V_{\sqrt{p}Q+\lambda}(\mu)$ is integrable.
The operator $h_{i,\lambda}(\mu)$ is integrable because 
\begin{align}
h_{i,\lambda}(\mu)|_{\mathcal{F}_{-\sqrt{p}\alpha+\lambda}}=(\alpha_i,\alpha+\hat{\lambda}+\mu)\operatorname{id}|_{\mathcal{F}_{-\sqrt{p}\alpha+\lambda}}.
\end{align}
 For $\beta\in Q$, let $A$ be a homogeneous vector in $\mathcal{F}_{\sqrt{p}\beta+\lambda}$ with the conformal weight $n$.
For $N\in\mathbb{N}$, by \eqref{(19.0)}, $f_i^NA$ is a homogeneous vector in $\mathcal{F}_{\sqrt{p}(\beta+N\alpha_i)+\lambda}$ with the conformal weight $n$. By \eqref{(14)}, $\Delta_{\sqrt{p}(\beta+N\alpha_i)+\lambda}$ is the quadratic function in $N$ with the positive quadratic coefficient $p$. It follows that for a sufficiently large $N$, we have $n<\Delta_{\sqrt{p}(\beta+N\alpha_i)+\lambda}$, and by Remark \ref{confwtrmk}, $f_i^NA=0$. Thus, $f_i$ is also integrable.
\end{proof}
By \eqref{(19.0)}, \eqref{infinitesimalisom}, \eqref{infinitesimalisom2}, we obtain the following corollary.
\begin{cor}\label{cor:ft41}
For $\lambda\in\Lambda$ and $\mu\in P$, the $B$-action on $V_{\sqrt{p}Q+\lambda}(\mu)$ defined by Theorem \ref{ft41} gives the $V_{\sqrt{p}Q}$-module automorphisms.
\end{cor}

\subsection{Homogeneous vector bundles $\xi_\lambda(\mu)$}\label{subsectvectbdle}
Let us recall Section \ref{subsect:notation}, \ref{subsect:P_Jsubmod}.
For $\lambda\in\Lambda$, we consider the homogeneous vector bundle $\xi_\lambda(\mu)$ over the flag variety $G/B$ (for flag varieties, cf., e.g., \cite{Bo,Kum1}) defined by
\begin{align}\label{(31)}
\xi_{\lambda}(\mu)=G\times_BV_{\sqrt{p}Q+\lambda}(\mu).
\end{align}
Here the action of $B$ on $V_{\sqrt{p}Q+\lambda}$ is given by Theorem \ref{ft41}.
When $\mu=0$, we use the letter $\xi_\lambda$ instead of $\xi_\lambda(0)$.
\begin{rmk}\label{holom}
We have $V_{\sqrt{p}Q+\lambda}=\bigoplus_{\Delta\in\mathbb{Q}}(V_{\sqrt{p}Q+\lambda})_{\Delta}$ and $\operatorname{dim}(V_{\sqrt{p}Q+\lambda})_{\Delta}<\infty$ for each $\Delta\in\mathbb{Q}$.
By \eqref{(19.0)}, the $B$-action on $V_{\sqrt{p}Q+\lambda}(\mu)$ preserves the conformal grading.
Thus, we have
\begin{align}\label{confdecomp}
\xi_{\lambda}(\mu)\simeq\bigoplus_{\Delta\in\mathbb{Q}}\xi_{\lambda}(\mu)_{\Delta}:=\bigoplus_{\Delta\in\mathbb{Q}}G\times_B(V_{\sqrt{p}Q+\lambda})_{\Delta}
\end{align}
as vector bundles. If necessary, we regard $\xi_{\lambda}(\mu)$ as the direct sum of holomorphic vector bundles $\xi_{\lambda}(\mu)_{\Delta}$ by this decomposition. 
\end{rmk}

Let $N_+$ be the unipotent subgroup of $G$ with the Lie algebra $\mathfrak{n}_+$. For $X\in G$, there exists an element $\sigma\in W$ such that $\sigma X\in N_+B$. For $\sigma\in W$, let
\begin{equation}
\tilde{U}_{\sigma}=\{X\in G\ |\ \sigma X\in N_+B\}
\end{equation}
and $U_{\sigma}=\tilde{U}_{\sigma}/B\subseteq G/B$.
Then we have the Schubert open covering 
\begin{equation}\label{(33.001)}
G/B=\bigcup_{\sigma\in W}U_{\sigma}
\end{equation}
 of $G/B$ such that for each $\sigma\in W$, there exists a unique fixed point of the $H$-action in $U_{\sigma}$. Unless otherwise stated, we use this open covering \eqref{(33.001)} throughout the present paper.
For $\sigma,\tau\in W$, let
\begin{align}\label{dfn:transfunc}
\varphi_{\sigma,\tau}^{\lambda,\mu}\colon U_{\sigma}\cap U_{\tau}\rightarrow B
\end{align}
be the transition function of $\xi_\lambda(\mu)$.
Then by the embedding 
\begin{align}\label{H0embedding}
H^0(\xi_\lambda(\mu))\hookrightarrow\bigoplus_{\alpha\in P_+\cap Q}H^0({\xi_\lambda}(\mu)|_{U_\sigma}),~s\mapsto(s|_{U_\sigma})_{\sigma\in W},
\end{align}
we have the $G$-module isomorphism
\begin{equation}\label{strofH0}
H^0(\xi_{\lambda}(\mu))\simeq\left\{(s_{\sigma})_{\sigma\in W}\in\bigoplus_{\sigma\in W}H^0(\xi_{\lambda}(\mu)|_{U_{\sigma}})\ |\  \substack{\text{For any $\sigma, \tau\in W$,}\\
{s_{\tau}}|_{U_{\sigma}\cap U_{\tau}}=\varphi_{\sigma,\tau}^{\lambda,\mu}\circ{s_{\sigma}}|_{U_{\sigma}\cap U_{\tau}}}\right\},
\end{equation}
where ${s_{\sigma}}|_{U_{\sigma}\cap U_{\tau}}$ and ${s_{\tau}}|_{U_{\sigma}\cap U_{\tau}}$ mean the restriction of $s_{\sigma}$ and $s_\tau$ to $U_{\sigma}\cap U_{\tau}$, respectively.

For $\sigma,\tau\in W$, $\mathcal{O}_{G/B}(U_\sigma)$ and $\mathcal{O}_{G/B}(U_\sigma\cap U_\tau)$ are regarded as commutative vertex algebras with trivial conformal vectors, respectively.
We give the vertex operator algebra structure on $H^0(\xi_0|_{U_{\sigma}})$ by the tensor products of $\mathcal{O}_{G/B}(U_\sigma)$ and $V_{\sqrt{p}Q}$.
For each $\sigma\in W$, denote by $1_\sigma$ the vacuum vector of $\mathcal{O}_{G/B}(U_\sigma)$, that is, the constant function $1$ on $U_\sigma$.
Similarly, we also introduce the vertex operator algebra structure on $H^0(\xi_{0}|_{U_{\sigma}\cap U_{\tau}})\simeq\mathcal{O}_{G/B}(U_\sigma\cap U_\tau)\otimes V_{\sqrt{p}Q}$ for $\sigma,\tau\in W$. 
The transition function $\varphi_{\sigma,\tau}^{0,0}$ defines a linear automorphism of $H^0(\xi_{0}|_{U_{\sigma}\cap U_{\tau}})$ by
\begin{align}\label{beforevoastr}
(\varphi_{\sigma,\tau}^{0,0}(f\otimes A))(x)=f(x)\varphi_{\sigma,\tau}^{0,0}(x)A,
\end{align}
where $x\in U_{\sigma}\cap U_{\tau}$, $f\in\mathcal{O}_{G/B}(U_\sigma\cap U_\tau)$, and $A\in V_{\sqrt{p}Q}$. 
By Corollary \ref{cor:ft41} and \eqref{beforevoastr}, 
for $f,g\in\mathcal{O}_{G/B}(U_\sigma\cap U_\tau)$, $A,B\in V_{\sqrt{p}Q}$, and $n\in\mathbb{Z}$, we have
\begin{align}
&{\varphi_{\sigma,\tau}^{0,0}}(f\otimes|0\rangle)=f\otimes|0\rangle,~{ \varphi_{\sigma,\tau}^{0,0}}(f\otimes\omega)=f\otimes\omega,\\
& {\varphi_{\sigma,\tau}^{0,0}}((f\otimes A)_{(n)}(g\otimes B))=({\varphi_{\sigma,\tau}^{0,0}}(f\otimes A))_{(n)}({\varphi_{\sigma,\tau}^{0,0}}(g\otimes B)).\label{32.99}
 \end{align}
 Thus, for $\sigma,\tau\in W$, the transition function $\varphi_{\sigma,\tau}^{0,0}$ defines a vertex operator algebra automorphism of $H^0(\xi_0|_{U_\sigma\cap U_\tau})$.

\begin{lemm}\label{voastr}
The sheaf cohomology $H^0(\xi_0)$ has a vertex operator algebra structure. Furthermore, the group $G$ acts on $H^0(\xi_0)$ as  automorphisms of vertex operator algebra.
\end{lemm}

\begin{proof}
For $(s_{\sigma})_{\sigma\in W},(t_{\sigma})_{\sigma\in W}\in H^0(\xi_0)$ and $n\in\mathbb{Z}$, set 
\begin{equation}\label{(33.1)}
((s_{\sigma})_{\sigma\in W})_{(n)}(t_{\sigma})_{\sigma\in W}=((s_{\sigma})_{(n)}t_{\sigma})_{\sigma\in W}\in\bigoplus_{\sigma\in W}H^0(\xi_0|_{U_{\sigma}}).
\end{equation}
We show that \eqref{(33.1)} defines the vertex operator algebra structure on $H^0(\xi_0)$.
First, because the action of $\mathfrak{b}$ on $V_{\sqrt{p}Q}$ annihilates $|0\rangle$ and $\omega$, the vacuum vector $(1_w\otimes|0\rangle)_{w\in W}$ and the conformal vector $(1_w\otimes\omega)_{w\in W}$ of $\bigoplus_{w\in W}\mathcal{O}(U_w)\otimes V_{\sqrt{p}Q}$ are contained in $H^0(\xi_0)$.
Second, we show that $H^0(\xi_0)$ is closed under the product \eqref{(33.1)}.
Since the transition function $\varphi_{\sigma,\tau}^{0,0}$ is an automorphism of $H^0(\xi_0|_{U_\sigma\cap U_\tau})$, for $(s_{w})_{w\in W},(t_{w})_{w\in W}\in H^0(\xi_0)$ and $n\in\mathbb{Z}$, we have 
\begin{align}
((s_{\tau})_{(n)}t_{\tau})|_{U_{\sigma}\cap U_{\tau}}&=(\varphi_{\sigma,\tau}^{0,0}\circ{s_{\sigma}}|_{U_{\sigma}\cap U_{\tau}})_{(n)}(\varphi_{\sigma,\tau}^{0,0}\circ t_{\sigma}|_{U_{\sigma}\cap U_{\tau}})\\
&=\varphi_{\sigma,\tau}^{0,0}\circ((s_{\sigma}|_{U_{\sigma}\cap U_{\tau}})_{(n)}t_{\sigma}|_{U_{\sigma}\cap U_{\tau}})\nonumber\\
&=\varphi_{\sigma,\tau}^{0,0}\circ((s_{\sigma})_{(n)}t_{\sigma})|_{U_{\sigma}\cap U_{\tau}}.\nonumber
\end{align}
Hence we have $((s_{w})_{w\in W})_{(n)}(t_{w})_{w\in W}\in H^0(\xi_0)$ by \eqref{strofH0}. 

We prove the last claim.
 By the $G$-module isomorphism \eqref{(31.1)},
 $(\mathcal{O}_G(G)\otimes V_{\sqrt{p}Q})^B$ inherits the vertex operator algebra structure from $H^0(\xi_0)$ by $(s_{(n)}t)(g)=s(g)_{(n)}t(g)$ for $s,t\in(\mathcal{O}_G(G)\otimes V_{\sqrt{p}Q})^B$, $g\in G$, $n\in\mathbb{Z}$. 
 By definition, \eqref{(31.1)} is the vertex operator algebra isomorphism.
 Moreover, by \eqref{(31.2)}, we have
\begin{align}
(g(s_{(n)}t))(g_0)=(s_{(n)}t)(g^{-1}g_0)=(s(g^{-1}g_0))_{(n)}t(g^{-1}g_0)=((gs)_{(n)}gt)(g_0)
\end{align}
for $g,g_0\in G$, $s,t\in(\mathcal{O}_G(G)\otimes V_{\sqrt{p}Q})^B$ and $n\in\mathbb{Z}$. Thus, the last claim is proved.
\end{proof}
\hypertarget{lastcomment}{In the same way as Lemma \ref{voastr}, for each $\lambda\in\Lambda$ and $\mu\in\mathfrak{h}^\ast$, $H^0(\xi_{\lambda}(\mu))$ is the $H^0(\xi_0)$-module and $G$ acts on $H^0(\xi_{\lambda}(\mu))$ as  $H^0(\xi_0)$-module automorphisms.}

\subsection{Narrow screening operators}
For $\sigma\in W$ and $\lambda\in\Lambda$, set
\begin{align}\label{(16.2)}
\sigma\ast\lambda=-\sqrt{p}\hat{\lambda}+\frac{1}{\sqrt{p}}(\sigma(\sqrt{p}\bar\lambda+\rho)-\rho)\in\frac{1}{\sqrt{p}}P.
\end{align}
Then \eqref{(16.2)} defines a $W$-action on $\Lambda$.
For $\lambda\in\Lambda$ and $\sigma\in W$, set
\begin{align}\label{(666)}
\epsilon_{\lambda}(\sigma)=\frac{1}{\sqrt{p}}(\sigma\ast\bar\lambda-\overline{\sigma\ast\lambda})\in P.
\end{align}
Note that $\epsilon_{\lambda}(\operatorname{id})=0$ for any $\lambda\in\Lambda$.
Note also that when $(\alpha_i,\alpha_j)=-1$, we have
\begin{align}\label{tuika15}
(\epsilon_\lambda(\sigma_i),\alpha_j)=1\Leftrightarrow (\sqrt{p}\bar\lambda,\alpha_i+\alpha_j)+1\geq p\Leftrightarrow (\epsilon_\lambda(\sigma_j),\alpha_i)=1.
\end{align}
It is straightforward to verify that
\begin{align}
&\epsilon_\lambda(\sigma_i)+\epsilon_{\sigma_i\ast\lambda}(\sigma_i)=-\alpha_i-\delta_{(\sqrt{p}\bar\lambda,\alpha_i),p-1}\alpha_i,\label{401-1}\\
&\epsilon_{\lambda}(\sigma)=\sigma_i(\epsilon_{\lambda}(\sigma_i\sigma)+\epsilon_{\sigma_i\sigma\ast\lambda}(\sigma_i)+\rho)-\rho-\delta_{(\sqrt{p}\overline{\sigma\ast\lambda},\alpha_i),p-1}\alpha_i.\label{(401)}
\end{align}
\begin{rmk}\label{beloweps}
When $(\sqrt{p}\bar\lambda,\alpha_i)\leq p-2$, we have 
\begin{align}\label{neew}
(\epsilon_\lambda(\sigma_i),\alpha_i)=(\epsilon_{\sigma_i\ast\lambda}(\sigma_i),\alpha_i)=-1.
\end{align}
On the other hand, when $(\sqrt{p}\bar\lambda,\alpha_i)=p-1$, we have
$\sigma_i\ast\lambda=\lambda-\sqrt{p}\alpha_i$ and 
\begin{align}\label{401+1}
\epsilon_\lambda(\sigma_i)=\epsilon_{\sigma_i\ast\lambda}(\sigma_i)=-\alpha_i.
\end{align}
\end{rmk}
\begin{rmk}\label{rmkbeloweps}
For any $\lambda\in\Lambda$, $\sigma\in W$, and $1\leq i\leq l$ such that $l(\sigma_i\sigma)=l(\sigma)+1$, we have $(\epsilon_{\lambda}(\sigma),\alpha_i)\geq 0$ because $\sqrt{p}\bar\lambda\in P_+$.
On the other hand, for any $1\leq i\leq l$ such that $l(\sigma_i\sigma)=l(\sigma)-1$, by \eqref{401-1} and \eqref{(401)}, we have $(\epsilon_{\lambda}(\sigma),\alpha_i)\leq -1$.
\end{rmk}
The following Lemma is proved in Appendix. 
\begin{lemm}\label{condequiv}
For the longest element $w_0$ in $W$, let us take a minimal expression of $w_0$ as $w_0=\sigma_{i_{l(w_0)}}\dots\sigma_{i_1}$.
Then for $\lambda\in\Lambda$, the following conditions are equivalent:
\begin{enumerate}
\item\label{condition1}
 For any $1\leq n\leq l(w_0)-1$, we have
$(\epsilon_{\lambda}(\sigma_{i_n}\dots\sigma_{i_1}),\alpha_{i_{n+1}})=0$.
\item\label{condition2}
 We have $(\sqrt{p}\bar\lambda+\rho,\theta)\leq p$.
\end{enumerate}
\end{lemm}

To define the {\it narrow screening operators} \eqref{narrowscreening1} and \eqref{narrowscreening2}, let us recall notations in Section \ref{subsect:lattice}.
 By \cite[Section 5]{BK}, the vector space $V_{\frac{1}{\sqrt{p}}P}$ has a generalized vertex algebra \cite{BK,DL} structure by 
\begin{align}
V_{\frac{1}{\sqrt{p}}P}(a,z)=a(z),~V_{\frac{1}{\sqrt{p}}P}(|\mu\rangle,z)=V_{|\mu\rangle}(z),
\end{align}
for $a\in\mathfrak{h}$ and $\mu\in\frac{1}{\sqrt{p}}P$.
Note that $V_{\sqrt{p}Q}$ is not a generalized vertex subalgebra of $V_{\frac{1}{\sqrt{p}}P}$, because $V_{|\sqrt{p}\alpha\rangle}(z)\not=|\sqrt{p}\alpha\rangle(z)$ in general.
To distinguish vertex operators of $V_{\frac{1}{\sqrt{p}}P}$ from that of $V_{\sqrt{p}Q}$, for $a\in V_{\frac{1}{\sqrt{p}}P}$, let us use the notation
\begin{align}
{V_{\frac{1}{\sqrt{p}}P}}(a,z)=\sum_{n\in\mathbb{Q}}a^{\text{gen}}_{(n)}z^{-n-1}.
\end{align}
Then by \cite[(27)]{BK}, for any $\mu,\nu\in\frac{1}{\sqrt{p}}P$, $a\in\mathcal{F}_{\mu}$, $b\in\mathcal{F}_{\nu}$, $v\in V_{\frac{1}{\sqrt{p}}P}$, we have 
\begin{align}\label{genborcherds}
(a^{\text{gen}}_{(m)}b^{\text{gen}}_{(n)}-(-1)^{(\mu,\nu)}b^{\text{gen}}_{(n)}a^{\text{gen}}_{(m)})v=\sum_{j\geq 0}\binom{m}{j}(a^{\text{gen}}_{(j)}b)^{\text{gen}}_{(m+n-j)}v.
\end{align} 
For $1\leq i\leq l$, let $F_{i,0}$ denote the operator in $\operatorname{Hom}_{\mathbb{C}}(V_{\sqrt{p}P},V_{\sqrt{p}P-\frac{1}{\sqrt{p}}\alpha_i})$ defined by
\begin{align}\label{narrowscreening1}
F_{i,0}=|-\frac{1}{\sqrt{p}}\alpha_i\rangle_{(0)}={|-\frac{1}{\sqrt{p}}\alpha_i\rangle}^{\text{gen}}_{(0)}\epsilon(-\frac{1}{\sqrt{p}}\alpha_i,\cdot).
\end{align}
Then it is straightforward to check using \eqref{genborcherds} and \cite[(5.4.14)]{Kac} that for any $x\in \ker {F_{i,0}}|_{V_{\sqrt{p}Q}}$ and $n\in\mathbb{Z}$, we have
\begin{align}\label{(18.5)}
[F_{i,0},x_{(n)}]=0.
\end{align}

\hypertarget{perp}{Denote by $\mathcal{F}_0^i$ the rank $1$ Heisenberg vertex algebra generated by $\alpha_i$, and $\mathcal{F}_0^{i,\perp}$ the rank $l-1$ Heisenberg vertex algebra generated by $\{\omega_j\}_{1\leq j\not=i\leq l}$.
Then for $\alpha\in Q$ and $\lambda\in\Lambda$, we have }
\begin{align}\label{alignperp}
\mathcal{F}_{-\sqrt{p}\alpha+\lambda}\simeq\mathcal{U}(\mathcal{F}_0^{i,\perp})\otimes\mathcal{U}(\mathcal{F}_0^i)|-\sqrt{p}\alpha+\lambda\rangle.
\end{align}
By \cite{NT,TK}, if $s_i:=(\sqrt{p}\bar\lambda,\alpha_i)\leq p-2$, then we have the narrow screening operator $F_{i,\lambda}\in\operatorname{Hom}_{\mathbb{C}}(\mathcal{U}(\mathcal{F}_0^i)|-\sqrt{p}\alpha+\lambda\rangle,\mathcal{U}(\mathcal{F}_0^i)|-\sqrt{p}\alpha+\sigma_i\ast\lambda\rangle)$ defined by
\begin{align}\label{narrowscreening2}
F_{i,\lambda}=\int_{[\Gamma_{s_i+1}]}|-\frac{1}{\sqrt{p}}\alpha_i\rangle(z_1)\dots |-\frac{1}{\sqrt{p}}\alpha_i\rangle(z_{s_i+1})\mathrm{d}z_1\dots\mathrm{d}z_{s_i+1},
\end{align}
where the cycle $[\Gamma_{s_i+1}]$ is given in \cite[Proposition 2.1]{NT}. By \cite[Theorem 2.7, 2.8]{NT}, we have $F_{i,\lambda}\not=0$.  
Because the field $|-\frac{1}{\sqrt{p}}\alpha_i\rangle(z)$ commutes with the action of $\mathcal{U}(\mathcal{F}_0^{i,\perp})$, $F_{i,\lambda}$ also defines the operator in $\operatorname{Hom}_{\mathbb{C}}(\mathcal{F}_{-\sqrt{p}\alpha+\lambda},\mathcal{F}_{-\sqrt{p}\alpha+\sigma_i\ast\lambda})$, and thus, in
$
\operatorname{Hom}_{\mathbb{C}}(V_{\sqrt{p}P+\bar\lambda},V_{\sqrt{p}P+\sigma_i\ast\bar\lambda}).
$
For convenience, for $\lambda\in\Lambda$ such that $(\sqrt{p}\bar\lambda,\alpha_i)=p-1$, we set $F_{i,\lambda}=0$.
For $1\leq i,j\leq l$, $\mu\in\mathfrak{h}^\ast$, $a\in\ker F_{i,0}|_{V_{\sqrt{p}Q}}$ and $n\in\mathbb{Z}$, it is straightforward to verify that 
\begin{align}
&[F_{j,\lambda},a_{(n)}]=[F_{j,\lambda}, L_n]=[F_{j,\lambda}, f_i]=0,\label{(19.1)}\\
&F_{j,\lambda}h_{i,\lambda}(\mu)=h_{i,\sigma_i\ast\lambda}(\mu+\epsilon_\lambda(\sigma_j))F_{j,\lambda}.\label{(19.11)}
\end{align}
In particular, $F_{i,\lambda}$ preserves the conformal grading.

By \eqref{(18.5)} and \eqref{(19.1)}, for any subset $J$ of $\Pi$, $\bigcap_{i\in J}\ker F_{i,0}|_{V_{\sqrt{p}Q}}$ is a vertex operator subalgebra of $V_{\sqrt{p}Q}$ with the conformal vector $\omega$.
Also, by \eqref{(19.1)}, for every $\lambda\in\Lambda$, $\bigcap_{i\in J}\ker F_{i,\lambda}|_{V_{\sqrt{p}Q+\lambda}}$ is a module over $\bigcap_{i\in J}\ker F_{i,0}|_{V_{\sqrt{p}Q}}$.

\begin{dfn}\label{logWalg}
The {\it logarithmic $W$-algebra associated with $\sqrt{p}Q$}
is the vertex operator subalgebra $W(p)_Q$ of $V_{\sqrt{p}Q}$
defined by
\begin{align*}
W(p)_Q=\bigcap_{i=1}^l\ker F_{i,0}|_{V_{\sqrt{p}Q}}.
\end{align*}
Also, for every $\lambda\in\Lambda$, we define the $W(p)_Q$-submodule $W(p,\lambda)_Q$ of $V_{\sqrt{p}Q+\lambda}$ by
\begin{align*}
W(p,\lambda)_Q=\bigcap_{i=1}^l\ker F_{i,\lambda}|_{V_{\sqrt{p}Q+\lambda}}.
\end{align*}
\end{dfn}
In particular, in the case where $\mathfrak{g}=\mathfrak{sl}_2$, $W(p)_Q$ is the {\it triplet $W$-algebra} (\cite{AM1}-\cite{AM3}, \cite{FGST1}-\cite{FGST3}, \cite{NT}, \cite{TW}).

\begin{dfn}\label{logWalgs}
\begin{enumerate}
\item
For any $1\leq i\leq l$ and $\lambda\in\Lambda$,
$W(p)_Q^i$ denotes the vertex operator subalgebra of $V_{\sqrt{p}Q}$ 
defined by
\begin{align*}
W(p)_Q^i=\ker F_{i,0}|_{V_{\sqrt{p}Q}}, 
\end{align*}
and $W(p,\lambda)_Q^i$ denotes the $W(p)_Q^i$-submodule of $V_{\sqrt{p}Q+\lambda}$
defined by
\begin{align*}
W(p,\lambda)_Q^i=\ker F_{i,\lambda}|_{V_{\sqrt{p}Q+\lambda}}.
\end{align*}
\item
More generally, for any subset $J\subseteq\Pi$ and $\lambda\in\Lambda$,  $W(p)_Q^J$ denotes the vertex operator subalgebra of $V_{\sqrt{p}Q}$ 
defined by
\begin{align*}
W(p)_Q^J=\bigcap_{i\in J}\ker F_{i,0}|_{V_{\sqrt{p}Q}},
\end{align*}
and $W(p,\lambda)_Q^J$  denotes the $W(p)_Q^J$-submodule of $V_{\sqrt{p}Q+\lambda}$ defined by
\begin{align*}
W(p,\lambda)_Q^J=\bigcap_{i\in J}\ker F_{i,\lambda}|_{V_{\sqrt{p}Q+\lambda}}. 
\end{align*}
\end{enumerate}
\end{dfn}

Note that  we have $W(p)_Q^\Pi=W(p)_Q$ and $W(p,\lambda)_Q^\Pi=W(p,\lambda)_Q$ for $J=\Pi$.
Note also that $W(p,\lambda)_Q^J$ is a $W(p)_Q^I$-module for $J\subseteq I\subseteq \Pi$.
By \eqref{(19.1)} and \eqref{(19.11)}, for any $J\subseteq\Pi$ and $\lambda\in\Lambda$, $W(p,\lambda)_Q^J$ is a $B$-submodule of $V_{\sqrt{p}Q+\lambda}(\mu)$ for the $B$-action in Theorem \ref{ft41}.
More generally, for any $\mu\in\mathfrak{h}^\ast$, 
$W(p,\lambda)_Q^J(\mu)$ (see \hyperlink{$V(\mu)$}{Section \ref{subsect:notation}}) is a $B$-submodule of $V_{\sqrt{p}Q+\lambda}(\mu)$ for the $B$-action in Theorem \ref{ft41}.

Similarly,
by \eqref{(19.0)} and \eqref{infinitesimalisom}, for any subset $J\subseteq \Pi$, $\bigcap_{i\in J}\ker f_i|_{\mathcal{F}_0}$ is a vertex operator subalgebra of $\mathcal{F}_0$ with the conformal vector $\omega$.

\begin{dfn}\label{dfnmathcalW}
Let $\mathcal{W}_0$ be the vertex operator subalgebra of $\mathcal{F}_0$ is defined by
\begin{align*}
\mathcal{W}_0=\bigcap_{i=1}^l\ker f_i|_{\mathcal{F}_0}.
\end{align*}
Also, for any $\lambda\in\Lambda$ and $\alpha\in Q$, 
let $\mathcal{W}_{-\sqrt{p}\alpha+\lambda}$  be
 the $\mathcal{W}_0$-submodule of $\mathcal{F}_{-\sqrt{p}\alpha+\lambda}$ defined by
\begin{align*}
\mathcal{W}_{-\sqrt{p}\alpha+\lambda}=\bigcap_{i=1}^l(\mathcal{U}(\ker f_i|_{\mathcal{F}_0})|-\sqrt{p}\alpha+\lambda\rangle).
\end{align*}
\end{dfn}

\begin{dfn}\label{dfnmathcalWs}
\begin{enumerate}
\item
For any $1\leq i\leq l$, $\alpha\in Q$ and $\lambda\in\Lambda$, $\mathcal{W}_0^i$ denotes the vertex operator subalgebra of $\mathcal{F}_0$ defined by
\begin{align*}
\mathcal{W}_0^i=\ker f_i|_{\mathcal{F}_0},
\end{align*}
 and $\mathcal{W}_{-\sqrt{p}\alpha+\lambda}^i$ denotes the $\mathcal{W}_0^i$-submodule of $\mathcal{F}_{-\sqrt{p}\alpha+\lambda}$ defined by
\begin{align*}
\mathcal{W}_{-\sqrt{p}\alpha+\lambda}^i=\mathcal{U}(\mathcal{W}_0^i)|-\sqrt{p}\alpha+\lambda\rangle.\end{align*}
\item
More generally, for any subset $J\subseteq\Pi$, $\alpha\in Q$ and $\lambda\in\Lambda$, $\mathcal{W}_0^J$ denotes the vertex operator subalgebra of $\mathcal{F}_0$ defined by
\begin{align*}
\mathcal{W}_0^J=\bigcap_{i\in J}\mathcal{W}_0^i,
\end{align*}
and $\mathcal{W}_{-\sqrt{p}\alpha+\lambda}^J$ denotes the $\mathcal{W}_0^J$-submodule of $\mathcal{F}_{-\sqrt{p}\alpha+\lambda}$ defined by
\begin{align*}
\mathcal{W}_{-\sqrt{p}\alpha+\lambda}^J=\bigcap_{i\in J}\mathcal{W}_{-\sqrt{p}\alpha+\lambda}^i.\end{align*}
\end{enumerate}
\end{dfn}
Note that we have $\mathcal{W}_0=\mathcal{W}_0^\Pi$ and $\mathcal{W}_{-\sqrt{p}\alpha+\lambda}=\mathcal{W}_{-\sqrt{p}\alpha+\lambda}^\Pi$ for $J=\Pi$.
Note also that $\mathcal{W}_{-\sqrt{p}\alpha+\lambda}^J$ is a $\mathcal{W}_0^I$-module for $J\subseteq I\subseteq\Pi$.
The following lemma follows from Theorem \ref{am1} \eqref{am1:modstr} in Section \ref{reviewam1}.
\begin{lemm}\label{lemm in new section}
For $J\subseteq I\subseteq \Pi$, $\mathcal{W}_0^I$ is a vertex operator subalgebra of $W(p)_Q^J$.
Moreover, for any $\lambda\in\Lambda$ and $\alpha\in Q$ such that $(\alpha+\hat\lambda,\alpha_i)\geq 0$ for any $i\in J$, $\mathcal{W}_{-\sqrt{p}\alpha+\lambda}^J$ is a $\mathcal{W}^I_0$-submodule of $W(p,\lambda)_Q^J$.
\end{lemm}

\begin{rmk}
We expect
that  
$\mathcal{U}(\mathcal{W}_0^J)|\sqrt{p}\alpha+\lambda\rangle=\mathcal{W}^J_{\sqrt{p}\alpha+\lambda}$ and is irreducible as a $\mathcal{W}^J_0$-module
for any $J\subseteq\Pi$ and $\lambda\in\Lambda$. 
Note that the inclusion 
$\mathcal{U}(\mathcal{W}_0^J)|\sqrt{p}\alpha+\lambda\rangle\subseteq\mathcal{W}^J_{\sqrt{p}\alpha+\lambda}$
is easy to see.
\end{rmk}

\subsection{A review of some results in \cite{AM1,NT}}\label{reviewam1}
Let us recall notation in \eqref{alignperp} and \hyperlink{$V(\mu)$}{Section \ref{subsect:notation}}.
For any $\lambda\in\Lambda$ and $1\leq j\leq l$, we have the following decomposition:
\begin{align}\label{aligndecomp}
V_{\sqrt{p}Q+\lambda}\simeq\mathcal{U}(\mathcal{F}_0^{j,\perp})\otimes(\bigoplus_{[\alpha]\in Q/\mathbb{Z}\alpha_j}\bigoplus_{n\in\mathbb{Z}}\mathcal{U}(\mathcal{F}_0^j)|-\sqrt{p}(\alpha+n\alpha_j)+\lambda\rangle),
\end{align}
where for $\alpha,\beta\in Q$, $[\alpha]=[\beta]$ means that $\beta\in\alpha+\mathbb{Z}\alpha_j$.

\begin{thm}\label{am1}
Let $1\leq j\leq l$ and $\lambda\in\Lambda$. 
\begin{enumerate}
\item\label{am1:exseq}
If $(\sqrt{p}\bar\lambda,\alpha_j)\leq p-2$, then for any $\mu\in P$, we have the exact sequence
\begin{align}
0\rightarrow W(p,\lambda)^j_Q(\mu)\hookrightarrow V_{\sqrt{p}Q+\lambda}(\mu)\xrightarrow{F_{j,\lambda}} W(p,\sigma_j\ast\lambda)^j_Q(\mu+\epsilon_{\lambda}(\sigma_j))\rightarrow 0\nonumber
\end{align}
of $B$-modules and $\mathcal{W}^j_0$-modules.

\item\label{am1:cosing}
Let $\alpha\in Q$ such that $(\alpha+\hat\lambda,\alpha_j)\geq-1$. If $(\sqrt{p}\bar\lambda,\alpha_j)\leq p-2$, then there exists a vector $w_{-\sqrt{p}\alpha+\lambda}\in\mathcal{F}_{-\sqrt{p}\alpha+\lambda}$ such that
\begin{align}
&F_{j,\lambda}w_{-\sqrt{p}\alpha+\lambda}=|-\sqrt{p}(\alpha+\hat\lambda-\epsilon_\lambda(\alpha_j))+\overline{\sigma_j\ast\lambda}\rangle,\nonumber\\
&f_j^{(\alpha+\hat\lambda,\alpha_j)+1}w_{-\sqrt{p}\alpha+\lambda}=|-\sqrt{p}(\sigma_j(\alpha+\hat\lambda)-\alpha_j)+\bar\lambda\rangle.\nonumber
\end{align}
\item\label{am1:modstr}
As $\mathcal{W}^j_0$-modules, we have
\begin{align*}
W(p,\lambda)_Q^j\simeq\bigoplus_{\substack{\alpha\in Q,\\(\alpha+\hat\lambda,\alpha_j)\geq 0}}\bigoplus_{m=0}^{(\alpha+\hat\lambda,\alpha_j)}f_j^m\mathcal{W}^j_{-\sqrt{p}\alpha+\lambda}.
\end{align*}
Moreover, when $(\sqrt{p}\bar\lambda,\alpha_j)\leq p-2$, as $\mathcal{W}^j_0$-module, $V_{\sqrt{p}Q+\lambda}$ is generated by a family of linearly independent vectors
\begin{align}
&\{f_j^m|-\sqrt{p}\alpha+\lambda\rangle~|~(\alpha+\hat\lambda,\alpha_j)\geq 0,~0\leq m\leq (\alpha+\hat\lambda,\alpha_j)\}\nonumber\\
\sqcup&\{f_j^mw_{-\sqrt{p}\alpha+\lambda}~|~(\alpha+\hat\lambda,\alpha_j)\geq 0,~0\leq m\leq (\alpha+\hat\lambda,\alpha_j)+1\}.\nonumber
\end{align}
\end{enumerate}
\end{thm}

 \begin{proof}
 Let $\mathfrak{g}_j\simeq\mathfrak{sl}_2$ be the Lie subalgebra of $\mathfrak{g}$ generated by $\{e_j,h_j,f_j\}$ (where we use the letters $e_j,f_j$ for the Chevalley generators of $\mathfrak{g}$). 
 Then applying \cite[Theorem 1.1, 1.2]{AM1} to $\mathfrak{g}_j$ and $\bigoplus_{n\in\mathbb{Z}}\mathcal{U}(\mathcal{F}_0^j)|-\sqrt{p}(\alpha+n\alpha_j)+\lambda\rangle$
 for each $[\alpha]\in Q/\mathbb{Z}\alpha_j$ in \eqref{aligndecomp}, we obtain Theorem \ref{am1} \eqref{am1:cosing}, \eqref{am1:modstr}.
On the other hand, applying \cite[Theorem 2.7]{NT}, \eqref{(19.1)}, \eqref{(19.11)} to the same objects, we obtain Theorem \ref{am1} \eqref{am1:exseq}.
 \end{proof}

By Theorem \ref{am1} \eqref{am1:cosing}, for $\lambda\in\Lambda$ such that $(\sqrt{p}\bar\lambda,\alpha_j)\leq p-2$ and $\beta\in Q+\hat\lambda$ such that $(\beta,\alpha_j)\geq 0$, we have 
\begin{align}\label{alignfrom42}
f_j^{(\beta,\alpha_j)}|-\sqrt{p}\beta+\bar\lambda\rangle=F_{j,\sigma_j\ast\lambda}|-\sqrt{p}(\sigma_j(\beta)+\epsilon_{\sigma_j\ast\lambda}(\sigma_j))+\overline{\sigma_j\ast\lambda}\rangle.
\end{align}

\begin{cor}\label{beforecoram1}
Let $1\leq j\leq l$, $\lambda\in\Lambda$, and $\alpha\in Q$ such that $(\alpha+\hat\lambda,\alpha_j)\geq 0$. Then for any $x\in\mathcal{W}_{-\sqrt{p}\alpha+\lambda}^j$, we have $f_j^{(\alpha+\hat\lambda,\alpha_j)+1}x=0$. 
\end{cor}
\begin{proof}
By definition, the action of $\mathcal{W}^j_0$ commutes with $f_j$. Thus, it is enough to show that $f_j^{(\alpha+\hat\lambda,\alpha_j)+1}|-\sqrt{p}\alpha+\lambda\rangle=0$.
If $(\sqrt{p}\bar\lambda,\alpha_j)\leq p-2$,
then we have
\begin{align}
f_j^{(\alpha+\hat\lambda,\alpha_j)+1}|-\sqrt{p}\alpha+\lambda\rangle
=&f_j(f_j^{(\alpha+\hat\lambda,\alpha_j)}|-\sqrt{p}(\alpha+\hat\lambda)+\bar\lambda\rangle)\nonumber\\
=&f_jF_{j,\sigma_j\ast\lambda}|-\sqrt{p}(\sigma_j(\alpha+\hat\lambda)+\epsilon_{\sigma\ast\lambda}(\sigma_j))+\overline{\sigma_j\ast\lambda}\rangle\nonumber\\
=&F_{j,\sigma_j\ast\lambda}f_j|-\sqrt{p}(\sigma_j(\alpha+\hat\lambda)+\epsilon_{\sigma\ast\lambda}(\sigma_j))+\overline{\sigma_j\ast\lambda}\rangle=0.\nonumber
\end{align}
Here the second equality follows from \eqref{alignfrom42}, the third equality follows from \eqref{(19.1)}, and the last equality follows from Remark \ref{confwtrmk} and the fact that
\begin{align}
\Delta_{-\sqrt{p}(\sigma_j(\alpha+\hat\lambda)+\epsilon_{\sigma_j\ast\lambda}(\sigma_j))+\overline{\sigma_j\ast\lambda}}<\Delta_{-\sqrt{p}(\sigma_j(\alpha+\hat\lambda)+\epsilon_{\sigma_j\ast\lambda}(\sigma_j)-\alpha_j)+\overline{\sigma_j\ast\lambda}}.\nonumber
\end{align}
If $(\sqrt{p}\bar\lambda,\alpha_j)=p-1$, then by definition, we have $V_{\sqrt{p}Q+\lambda}=W(p,\lambda)_Q^j$.
By Theorem \ref{am1} \eqref{am1:modstr} and the fact that
$\Delta_{-\sqrt{p}\alpha+\lambda}=\Delta_{-\sqrt{p}(\sigma_j(\alpha+\hat\lambda))+\bar\lambda}$,
we have 
\begin{align}\label{alignp-1case}
f_j^{(\alpha+\hat\lambda,\alpha_j)}|-\sqrt{p}\alpha+\lambda\rangle=a|-\sqrt{p}(\sigma_j(\alpha+\hat\lambda))+\bar\lambda\rangle
\end{align}
for some $a\in\mathbb{C}^\times$.
Then we have 
\begin{align}
f_j^{(\alpha+\hat\lambda,\alpha_j)+1}|-\sqrt{p}\alpha+\lambda\rangle=af_j|-\sqrt{p}(\sigma_j(\alpha+\hat\lambda))+\bar\lambda\rangle=0,
\end{align}
where the first equality follows from \eqref{alignp-1case}, and the second equality follows from Remark \ref{confwtrmk} and the fact that
$\Delta_{-\sqrt{p}(\sigma_j(\alpha+\hat\lambda))+\bar\lambda}<\Delta_{-\sqrt{p}(\sigma_j(\alpha+\hat\lambda)-\alpha_j)+\bar\lambda}$.
\end{proof}

\begin{cor}\label{coram1}
Let $1\leq j\leq l$, $\lambda\in\Lambda$, and $\alpha\in Q$ such that $(\alpha+\hat\lambda,\alpha_j)\geq 0$. Then:
\begin{enumerate}
\item\label{coram1.1}
For $v\in\mathcal{F}_{-\sqrt{p}\alpha+\lambda}$, we have $v=0$ if and only if $f_j^{(\alpha+\hat\lambda,\alpha_j)}v=0$.
\item\label{coram1.2}
For $v\in W(p,\lambda)_Q^j\cap\mathcal{F}_{-\sqrt{p}\alpha+\lambda}$, we have
$v\in\bigoplus_{m=0}^nf_j^m\mathcal{W}^j_{-\sqrt{p}(\alpha+m\alpha_j)+\lambda}$
if and only if $f_j^{(\alpha+\hat\lambda,\alpha_j)+m+1}v=0$.
\end{enumerate}
\end{cor}
\begin{proof}
Corollary \ref{coram1} \eqref{coram1.1} follows from Theorem \ref{am1} \eqref{am1:modstr}.
The necessity of the condition $f_j^{(\alpha+\hat\lambda,\alpha_j)+m+1}v=0$ in Corollary \ref{coram1} \eqref{coram1.2} follows from Corollary \ref{beforecoram1}, and 
the converse follows from Theorem \ref{am1} \eqref{am1:modstr}.
\end{proof}

We prove lemmas needed in Section \ref{lastsection}.
For $\beta\in P_+$, denote by $y_\beta$ and $y'_\beta$ the highest weight vector and the lowest weight vector of $\mathcal{R}_\beta$, respectively.

\begin{lemm}\label{newlemm3.4}
For any $\alpha\in P_+\cap Q$ and $\hat\lambda\in\hat\Lambda$, there exists an element ${u}_{\alpha+\hat\lambda}\in\mathcal{U}(\mathfrak{n}_-)$ such that for any nonzero vector $x$ in $\mathcal{F}_{-\sqrt{p}\alpha+\lambda}$, we have ${u}_{\alpha+\hat\lambda}x\not=0$ and ${u}_{\alpha+\hat\lambda}y_{\alpha+\hat\lambda}\in\mathbb{C}^\times y'_{\alpha+\hat\lambda}$. 
\end{lemm}
\begin{proof}
By Corollary \ref{coram1} \eqref{coram1.1}, we have $f_i^{(\alpha+\hat\lambda,\alpha_i)}x\not=0$ for any $1\leq i\leq l$. 
Similarly, for $1\leq j\leq l$ such that $(\sigma_i(\alpha+\hat\lambda),\alpha_j)\geq 0$, we have $f_j^{(\sigma_i(\alpha+\hat\lambda),\alpha_j)}f_i^{(\alpha+\hat\lambda,\alpha_i)}x\not=0$. By repeating this process finitely many times, 
for some ${u}_{\alpha+\hat\lambda}\in\mathcal{U}(\mathfrak{n}_-)$ such that ${u}_{\alpha+\hat\lambda}y_{\alpha+\hat\lambda}\in\mathbb{C}^\times y'_{\alpha+\hat\lambda}$, we have ${u}_{\alpha+\hat\lambda}x\not=0$.
\end{proof}

\begin{lemm}\label{newlemm3.9}
Let $x$ be a nonzero vector in $\mathcal{F}_{-\sqrt{p}\alpha+\lambda}$.
If there exists a $B$-module homomorphism $\Phi\colon\mathcal{R}_{\alpha+\hat\lambda}\rightarrow\mathcal{U}(\mathfrak{b})x$ that sends $y_{\alpha+\hat\lambda}$ to $x$, then $\Phi$ is an isomorphism.
In particular,
$\mathcal{U}(\mathfrak{b})x$ is a $G$-submodule of $V_{\sqrt{p}Q+\lambda}$
in the sense of Definition \ref{dfnGsubmodule}.
\end{lemm}
\begin{proof}
Since $\Phi(y_{\alpha+\hat\lambda})=x$, if $uy_{\alpha+\hat\lambda}=0$ for some $u\in\mathcal{U}(\mathfrak{b})$, then $ux=0$. By contraposition, $\Phi$ is sunjective.
We show that $\Phi$ is injective.
Let $ux=0$ for some $u\in\mathcal{U}(\mathfrak{b})$. 
If $uy_{\alpha+\hat\lambda}\not=0$, then because $\mathcal{R}_{\alpha+\hat\lambda}$ is 
an irreducible $\mathfrak{g}$-module, for some $u'\in\mathcal{U}(\mathfrak{n}_-)$, we have $u'uy_{\alpha+\hat\lambda}\in\mathbb{C}^\times y'_{\alpha+\hat\lambda}$. In other words, we have $u'uy_{\alpha+\hat\lambda}=a{u}_{\alpha+\hat\lambda}y_{\alpha+\hat\lambda}$ for some $a\in\mathbb{C}^\times$ and ${u}_{\alpha+\hat\lambda}$ in Lemma \ref{newlemm3.4}.
Then we have 
$0=u'ux=a{u}_{\alpha+\hat\lambda}x$, and by Lemma \ref{newlemm3.4}, we have $x=0$.
It contradicts the assumption, and thus, we have $uy_{\alpha+\hat\lambda}=0$.
\end{proof}

\section{Proofs of main results}\label{sectionproofs}
\subsection{The action of $P_j$ on $W(p,\lambda)_Q^j$}\label{parabolic}
Let us recall Definition \ref{logWalgs} and \ref{dfnmathcalWs}.
We fix an integer $1\leq j\leq l$ and $\lambda\in\Lambda$ throughout this subsection.
In this subsection, we show that $W(p,\lambda)_Q^j$ is a $P_j$-submodule of  $V_{\sqrt{p}Q+\lambda}$ in the sense of Definition \ref{dfnGsubmodule}.
We define the linear operator $e_j\in\operatorname{End}_{\mathbb{C}}(W(p,\lambda)_Q^j)$ by
\begin{equation}\label{(57.02)}
e_jf_j^NA_{\beta}=c_N^{\beta}f_j^{N-1}A_{\beta},
\end{equation}
where $\beta\in Q+\hat\lambda$ such that $(\beta,\alpha_j)\geq 0$, $c_N^{\beta}=N((\beta,\alpha_j)-N+1)$, and $A_{\beta}\in\mathcal{W}^j_{-\sqrt{p}\beta+\bar\lambda}$. By Theorem \ref{am1} \eqref{am1:modstr}, $e_j$ is well-defined.

\begin{lemm}\label{lemm:pre}
Let $\beta\in Q+\hat\lambda$ such that $(\beta,\alpha_j)\geq 0$, $A_{\beta}\in\mathcal{W}^j_{-\sqrt{p}\beta+\bar\lambda}$, and $1\leq i\leq l$ such that $i\not=j$. Then we have 
\begin{equation}\label{(57.1)}
f_iA_{\beta}\in\mathcal{W}^j_{-\sqrt{p}(\beta-\alpha_i)+\bar\lambda}.
\end{equation}
On the other hand, in the case where $(\alpha_i,\alpha_j)=-1$, we have
\begin{equation}\label{(57.2)}
[f_i,f_j]A_{\beta}=B^0_{\beta-\alpha_i-\alpha_j}+f_jB^1_{\beta-\alpha_i}
\end{equation}
for some $B^0_{\beta-\alpha_i-\alpha_j}\in\mathcal{W}^j_{-\sqrt{p}(\beta-\alpha_i-\alpha_j)+\bar\lambda}$ and $B^1_{\beta-\alpha_i}\in\mathcal{W}^j_{-\sqrt{p}(\beta-\alpha_i)+\bar\lambda}$.
\end{lemm}

\begin{proof}
Since $f_iA_{\beta}\in W(p,\lambda)_Q^j\cap\mathcal{F}_{-\sqrt{p}(\beta-\alpha_j)+\bar\lambda}$, by Theorem \ref{am1} \eqref{am1:modstr}, we have
\begin{align}\label{(57.21)}
f_iA_{\beta}=\sum_{m\geq 0}f_j^mA^m_{\beta-\alpha_i+m\alpha_j}
\end{align}
for some $A^m_{\beta-\alpha_i+m\alpha_j}\in\mathcal{W}^j_{-\sqrt{p}(\beta-\alpha_i+m\alpha_j)+\bar\lambda}$. 

First we consider the case of $(\alpha_i,\alpha_j)=0$. Since $[f_i,f_j]=0$, we have
\begin{equation}\label{tuika11}
f_j^nf_iA_{\beta}=f_if_j^nA_{\beta}=0
\end{equation}
for $n\geq(\beta,\alpha_j)+1$, where the second equality follows from Corollary \ref{coram1} \eqref{coram1.2}.
By multiplying both sides of \eqref{(57.21)} by $f_j^{(\beta,\alpha_j)+1}$, by \eqref{tuika11}, we have
\begin{equation}
0=f_j^{(\beta,\alpha_j)+1}f_iA_{\beta}=\sum_{m\geq0}f_j^{m+(\beta,\alpha_j)+1}A^m_{\beta-\alpha_i+m\alpha_j}.
\end{equation}
By Corollary \ref{coram1} \eqref{coram1.2}, we have $f_j^{(\beta,\alpha_j)+1}A^0_{\beta-\alpha_i}=0$.
However, since $m+(\beta,\alpha_j)+1\leq(\beta-\alpha_i+m\alpha_j,\alpha_j)$ for $m\geq1$, by Theorem \ref{am1} \eqref{am1:modstr}, 
$A^m_{\beta-\alpha_i+m\alpha_j}$ must be $0$
for $m\geq 1$. Hence $f_iA_{\beta}\in\mathcal{W}^j_{-\sqrt{p}(\beta-\alpha_i)+\bar\lambda}$, by \eqref{(57.21)}.

Next we consider the case of $(\alpha_i,\alpha_j)=-1$. Since $[f_j,[f_i,f_j]]=0$, we have
\begin{equation}\label{tuika12}
f_j^nf_iA_{\beta}=(n[f_j,f_i]+f_if_j)f_j^{n-1}A_{\beta}=0
\end{equation}
for $n\geq(\beta,\alpha_j)+2$, where the second equality follows from Corollary \ref{coram1} \eqref{coram1.2}.
By multiplying both sides of \eqref{(57.21)} by $f_j^{(\beta,\alpha_j)+2}$, by \eqref{tuika12}, we have
\begin{equation}
0=f_j^{(\beta,\alpha_j)+2}f_iA_{\beta}=\sum_{m\geq0}f_j^{m+(\beta,\alpha_j)+2}A^m_{\beta-\alpha_i+m\alpha_j}.
\end{equation}
In the same manner as above, we have $f_j^{(\beta,\alpha_j)+2}A^0_{\beta-\alpha_i}=0$ and $A^m_{\beta-\alpha_i+m\alpha_j}=0$
for $m\geq 1$. Hence $f_iA_{\beta}\in\mathcal{W}^j_{-\sqrt{p}(\beta-\alpha_i)+\bar\lambda}$, by \eqref{(57.21)}.

 Finally, let us prove \eqref{(57.2)}. By Theorem \ref{am1} \eqref{am1:modstr}, we have
 \begin{equation}\label{(57.3)}
[f_i,f_j]A_{\beta}=\sum_{m\geq 0}f_j^mB^m_{\beta-\alpha_i-(1-m)\alpha_j},
 \end{equation}
 for some $B^m_{\beta-\alpha_i-(1-m)\alpha_j}\in\mathcal{W}^j_{-\sqrt{p}(\beta-\alpha_i-(1-m)\alpha_j)+\bar\lambda}$. Since $[f_j,[f_i,f_j]]=0$, we have
 \begin{align}\label{tuika13}
 f_j^n[f_i,f_j]A_{\beta}=[f_i,f_j]f_j^nA_\beta=0
 \end{align}
 for $n\geq(\beta,\alpha_j)+1$, where the second equality follows from Corollary \ref{coram1} \eqref{coram1.2}.
 By multiplying both sides of \eqref{(57.3)} by $f_j^{(\beta,\alpha_j)+1}$, by \eqref{tuika13}, we have
 \begin{align}
 0=f_j^{(\beta,\alpha_j)+1}[f_i,f_j]A_{\beta}=\sum_{m\geq 0}f_j^{m+(\beta,\alpha_j)+1}B^m_{\beta-\alpha_i-(1-m)\alpha_j}.
\end{align}
In the same manner as above, we have $B^m_{\beta-\alpha_i-(1-m)\alpha_j}=0$ for $m\geq 2$.
Hence $[f_i,f_j]A_{\beta}=B^0_{\beta-\alpha_i-\alpha_j}+f_jB^1_{\beta-\alpha_i}$, by \eqref{(57.3)}.
\end{proof}

\begin{lemm}\label{lemm:serre}
For $1\leq i \leq l$, we have $[e_j, f_i]=\delta_{ij}h_{i,\lambda}$ on $W(p,\lambda)_Q^j$. In particular, $W(p,\lambda)_Q^j$ is a $P_j$-submodule of the $B$-module $V_{\sqrt{p}Q+\lambda}$.
\end{lemm}
\begin{proof}
By Theorem \ref{am1} \eqref{am1:modstr}, it is enough to show that 
\begin{align}\label{assertionalign1}
[e_j, f_i]f_j^NA_\beta=\delta_{ij}h_{i,\lambda}f_j^NA_\beta
\end{align}
for any $\beta\in Q+\hat\lambda$ such that $(\beta,\alpha_j)\geq 0$, $A_\beta\in\mathcal{W}^j_{-\sqrt{p}\beta+\bar\lambda}$, and $0\leq N\leq(\beta,\alpha_j)$.

We divide the proof into several cases.
In the case of $i=j$, \eqref{assertionalign1} follows from the definition.
Second, we consider the case of $(\alpha_i,\alpha_j)=0$.
By Lemma \ref{lemm:pre}, we have
\begin{align}
f_iA_\beta=A^0_{\beta-\alpha_i}
\end{align}
for some $A^0_{\beta-\alpha_i}\in\mathcal{W}^j_{-\sqrt{p}(\beta-\alpha_i)+\bar\lambda}$.
Then we have
\begin{align}
(f_ie_j-e_jf_i)f_j^NA_{\beta}&=c_N^{\beta}f_if_j^{N-1}A_{\beta}-e_jf_j^Nf_iA_{\beta}\\
&=c_N^{\beta}f_j^{N-1}f_iA_{\beta}-e_jf_j^Nf_iA_{\beta}\nonumber\\
&=c_N^{\beta}f_j^{N-1}A^0_{\beta-\alpha_i}-e_jf_j^NA^0_{\beta-\alpha_i}\nonumber\\
&=(c_N^{\beta}-c_N^{\beta-\alpha_i})f_j^{N-1}A^0_{\beta-\alpha_i}=0,\nonumber
\end{align}
where the first and fourth equality follows from \eqref{(57.02)}, and the second equality follows from the fact that $[f_i,f_j]=0$.
Finally, we consider the case of $(\alpha_i,\alpha_j)=-1$. 
By Lemma \ref{lemm:pre}, we have
\begin{align}
[f_i,f_j]A_\beta=C^0_{\beta-\alpha_i-\alpha_j}+f_jC^1_{\beta-\alpha_i},~f_iA_\beta=D^0_{\beta-\alpha_i}
\end{align}
for some $C^0_{\beta-\alpha_i-\alpha_j}\in\mathcal{W}^j_{-\sqrt{p}(\beta-\alpha_i-\alpha_j)+\bar\lambda}$ and $C^1_{\beta-\alpha_i}, D^0_{\beta-\alpha_i}\in\mathcal{W}^j_{-\sqrt{p}(\beta-\alpha_i)+\bar\lambda}$.
Then we have
\begin{align}\label{beforealign}
&e_jf_if_j^NA_{\beta}\\
=&e_j(Nf_j^{N-1}[f_i,f_j]+f_j^Nf_i)A_{\beta}\nonumber\\
=&Ne_jf_j^{N-1}C^0_{\beta-\alpha_i-\alpha_j}+Ne_jf_j^NC^1_{\beta-\alpha_i}+e_jf_j^ND^0_{\beta-\alpha_i}\nonumber\\
=&Nc_{N-1}^{\beta-\alpha_i-\alpha_j}f_j^{N-2}C^0_{\beta-\alpha_i-\alpha_j}+Nc_{N}^{\beta-\alpha_i}f_j^{N-1}C^1_{\beta-\alpha_i}+c_{N}^{\beta-\alpha_i}f_j^{N-1}D^0_{\beta-\alpha_i},\nonumber
\end{align}
where the first equality follows from the Serre relation, and the third equality follows from \eqref{(57.02)}.
In the same manner as \eqref{beforealign}, we have
\begin{align}\label{afteralign}
&f_ie_jf_j^NA_{\beta}\\
=&(N-1)c_{N}^{\beta}f_j^{N-2}C^0_{\beta-\alpha_i-\alpha_j}+(N-1)c_{N}^{\beta}f_j^{N-1}C^1_{\beta-\alpha_i}+c_{N}^{\beta}f_j^{N-1}D^0_{\beta-\alpha_i}.\nonumber
\end{align}
By \eqref{beforealign} and \eqref{afteralign}, we obtain that
\begin{align}
[e_j, f_i]f_j^NA_{\beta}&=N(\beta-\alpha_i,\alpha_j)f_j^{N-1}C^1_{\beta-\alpha_i}+Nf_j^{N-1}D^0_{\beta-\alpha_i}\\
&=Nf_j^{N-1}[e_j, f_i]f_jA_{\beta},\nonumber
\end{align}
Hence it is enough to show that $[e_j, f_i]f_jA_{\beta}=0$.
Note that it is sufficient to show that $f_j^{(\beta-\alpha_i,\alpha_j)}[e_j, f_i]f_jA_{\beta}=0$, by Corollary \ref{coram1} \eqref{coram1.1}.
By Lemma \ref{lemm:pre}, we have
\begin{align}\label{m1align}
f_if_jA_{\beta}=A^0_{\beta-\alpha_i-\alpha_j}+f_jA^1_{\beta-\alpha_i},~f_iA_{\beta}=B^0_{\beta-\alpha_i}
\end{align}
for some $A^0_{\beta-\alpha_i-\alpha_j}\in\mathcal{W}^j_{-\sqrt{p}(\beta-\alpha_i-\alpha_j)+\bar\lambda}$ and  $A^1_{\beta-\alpha_i}, B^0_{\beta-\alpha_i}\in\mathcal{W}^j_{-\sqrt{p}(\beta-\alpha_i)+\bar\lambda}$. Moreover, by multiplying $f_if_jA_\beta$ in \eqref{m1align} by $e_j$, we have
\begin{equation}\label{m1align2}
[e_j, f_i]f_jA_{\beta}=c_1^{\beta-\alpha_i}A^1_{\beta-\alpha_i}-c_1^{\beta}B^0_{\beta-\alpha_i}.
\end{equation}
By combining \eqref{m1align} with \eqref{m1align2}, we have
\begin{align}\label{strangeway}
f_j[e_j, f_i]f_jA_{\beta}=c_1^{\beta-\alpha_i}f_if_jA_{\beta}-c_1^{\beta}f_jf_iA_{\beta}-c^{\beta-\alpha_i}_1A^0_{\beta-\alpha_i-\alpha_j}.
\end{align}
Then we have
\begin{align}
f_j^{(\beta-\alpha_i,\alpha_j)}[e_j, f_i]f_jA_{\beta}=&f_j^{(\beta,\alpha_j)}(f_j[e_j, f_i]f_jA_{\beta})\nonumber\\
=&f_j^{(\beta-\alpha_i-\alpha_j,\alpha_j)+1}(c_1^{\beta-\alpha_i}f_if_jA_{\beta}-c_1^{\beta}f_jf_iA_{\beta}-c^{\beta-\alpha_i}_1A^0_{\beta-\alpha_i-\alpha_j})\nonumber\\
=&((\beta,\alpha_j)c_1^{\beta-\alpha_i}-((\beta,\alpha_j)+1)c_1^{\beta})[f_j,f_i]f_j^{(\beta,\alpha_j)}A_{\beta}=0,\nonumber
\end{align}
where the second equality follows from \eqref{strangeway}, and the third equality follows from Corollary \ref{coram1} \eqref{coram1.2}. Thus the claim is proved.
\end{proof}

\subsection{Generalization of Lemma \ref{lemm:serre}}\label{almostproof}
Let us recall Definition \ref{logWalgs}, \ref{dfnmathcalWs}.
For $(i,j,\lambda)\in \Pi\times\Pi\times \Lambda$, we consider the following condition:
\begin{align}\label{novelcond}
\text{$\epsilon_\lambda(\sigma_j)=-\alpha_j$ or $(\epsilon_\lambda(\sigma_j),\alpha_i)=-\delta_{ij}$
(where $\epsilon_\lambda(\sigma_j)$ is given in \eqref{(666)}).}
\end{align}
For a subset $J\subseteq\Pi$, we also consider the following condition:
\begin{align}\label{novelcond2}
\text{$(i,j,\lambda)$ satisfies \eqref{novelcond} for any $(i,j)\in J\times J$}.
\end{align}
When $J=\Pi$, the condition \eqref{novelcond2} is stated as the following:
\begin{align}\label{novelcond3}
\text{$(i,j,\lambda)$ satisfies \eqref{novelcond} for any $(i,j)\in\{1,\dots,l\}\times\{1,\dots,l\}$.}
\end{align}
\begin{lemm}\label{special}
Let $1\leq i,j\leq l$ and $\lambda\in\Lambda$. If $(i,j,\lambda)$ satisfies \eqref{novelcond}, then $W(p,\lambda)_Q^{\{i,j\}}$ is closed under the operator $e_i$ in \eqref{(57.02)}.
\end{lemm}
\begin{proof}
If $\epsilon_\lambda(\sigma_j)=-\alpha_j$ or $i=j$, then we have $W(p,\lambda)_Q^{\{i,j\}}=W(p,\lambda)_Q^i$ and the assertion is clear. Let us assume that $i\not=j$ and $(\epsilon_\lambda(\sigma_j),\alpha_i)=0$.
It is enough to show that if $f_ia$ is a nonzero vector in $W(p,\lambda)_Q^{\{i,j\}}$, then we have $a\in W(p,\lambda)_Q^{\{i,j\}}$. By Theorem \ref{am1} \eqref{am1:modstr}, it is enough to consider the case where
\begin{align}
a=f_i^nA_\beta\in W(p,\lambda)_Q^i
\end{align}
 for some $A_{\beta}\in\mathcal{W}^i_{-\sqrt{p}\beta+\bar\lambda}$, $(\beta,\alpha_i)\geq0$, and $0\leq n<(\beta,\alpha_i)$.
By the assumption that $f_ia=f_i^{n+1}A_\beta\in W(p,\lambda)_Q^j$, we have 
\begin{align}
f_i^{(\beta,\alpha_i)}F_{j,\lambda}A_{\beta}=F_{j,\lambda}f_i^{(\beta,\alpha_i)}A_{\beta}=0,
\end{align}
where the first equality follows from \eqref{(19.1)}.
By Corollary \ref{coram1} \eqref{coram1.1} and the assumption that $(\epsilon_{\lambda}(\sigma_j),\alpha_i)=0$, we have $F_{j,\lambda}A_{\beta}=0$, that is, $A_\beta\in W(p,\lambda)^j_Q$. In particular, we have $a\in W(p,\lambda)^j_Q$.
\end{proof}

\begin{thm}\label{parab}
If and only if $(J,\lambda)$ satisfies \eqref{novelcond2}, then $W(p,\lambda)_Q^J$ is a $P_J$-submodule of $V_{\sqrt{p}Q+\lambda}$. 
\end{thm}
\begin{proof}
First, let us assume that $(J,\lambda)$ satisfies \eqref{novelcond2}. By using Lemma \ref{special} repeatedly, it follows that the subspace $W(p,\lambda)_Q^J\subseteq V_{\sqrt{p}Q+\lambda}$ is closed under the operators $\{e_i\}_{i\in J}$ and the $B$-action. By Lemma \ref{lemm:serre}, to show that $W(p,\lambda)_Q^J$ is a $P_J$-submodule of $V_{\sqrt{p}Q+\lambda}$, it is enough to verify that the operators $\{e_i\}_{i\in J}$ satisfy the Serre relations on $W(p,\lambda)_Q^J$. By Theorem \ref{am1} \eqref{am1:modstr}, it is enough to show that
$\operatorname{ad}(e_i)^{1-c_{ij}}e_jf_i^NA_{\beta}=0$
for any $(i,j)\in J\times J$ such that $i\not=j$, $f_i^NA_{\beta}\in W(p,\lambda)_Q^{\{i,j\}}$, $A_\beta\in\mathcal{W}^i_{-\sqrt{p}\beta+\bar\lambda}$. 
First, let us assume that $(\alpha_i,\alpha_j)=0$. Then by Lemma \ref{lemm:serre}, we have
\begin{align}\label{t1}
f_i^{(\beta+\alpha_j,\alpha_i)+1}e_jA_\beta=e_jf_i^{(\beta,\alpha_i)+1}A_\beta=0.
\end{align}
By \eqref{t1} and Corollary \ref{coram1} \eqref{coram1.2}, we have $e_jA_\beta\in\mathcal{W}^i_{-\sqrt{p}(\beta+\alpha_j)+\bar\lambda}$.
Thus we have
\begin{align}
[e_i,e_j]f_i^NA_{\beta}=e_if_i^Ne_jA_{\beta}-c^{\beta}_Nf_i^{N-1}e_jA_{\beta}=(c^{\beta+\alpha_j}_N-c^{\beta}_N)f_i^{N-1}e_jA_\beta=0.\nonumber
\end{align}
Next, let us assume that $(\alpha_i,\alpha_j)=-1$.
Then by Lemma \ref{lemm:serre}, we have
\begin{align}\label{t2}
f_i^{(\beta+\alpha_j,\alpha_i)+2}e_jA_\beta=e_jf_i^{(\beta,\alpha_i)+1}A_\beta=0.
\end{align}
By \eqref{t2} and Corollary \ref{coram1} \eqref{coram1.2}, we have $e_jA_\beta=A_{\beta+\alpha_j}+f_iA_{\beta+\alpha_j+\alpha_i}$ for some $A_{\beta+\alpha_j}\in\mathcal{W}^i_{-\sqrt{p}(\beta+\alpha_j)+\bar\lambda}$ and $A_{\beta+\alpha_j+\alpha_i}\in\mathcal{W}^i_{-\sqrt{p}(\beta+\alpha_j+\alpha_i)+\bar\lambda}$.
Thus we have
\begin{align}
[e_i[e_i,e_j]]f_i^NA_{\beta}&=(e_i^2e_j-2e_ie_je_i+e_je_i^2)f_i^NA_{\beta}\nonumber\\
&=(e_i^2f_i^N-2c^{\beta}_Ne_if_i^{N-1}+c^{\beta}_Nc^{\beta}_{N-1})e_jA_{\beta}\nonumber\\
&=(e_i^2f_i^N-2c^{\beta}_Ne_if_i^{N-1}+c^{\beta}_Nc^{\beta}_{N-1})(A_{\beta+\alpha_j}+f_iA_{\beta+\alpha_j+\alpha_i})\nonumber\\
&=(c^{\beta+\alpha_j}_Nc^{\beta+\alpha_j}_{N-1}-2c^{\beta}_Nc^{\beta+\alpha_j}_{N-1}+c^{\beta}_Nc^{\beta}_{N-1})f_i^{N-2}A_{\beta+\alpha_j}\nonumber\\
&+(c^{\beta+\alpha_j+\alpha_i}_{N+1}c^{\beta+\alpha_j+\alpha_i}_{N}-2c^{\beta}_Nc^{\beta+\alpha_j+\alpha_i}_N+c^{\beta}_Nc^{\beta}_{N-1})f_i^{N-1}A_{\beta+\alpha_j+\alpha_i}=0.\nonumber
\end{align}
Therefore, if $(J,\lambda)$ satisfies \eqref{novelcond2}, then $W(p,\lambda)_Q^J$ is a $P_J$-submodule of $V_{\sqrt{p}Q+\lambda}$. 

We prove the converse of the assertion above. Let us assume that $(J,\lambda)$ does not satisfy \eqref{novelcond2}. Then there exists a pair $(j,i)\in J\times J$ such that $i\not=j$, $\epsilon_{\lambda}(\sigma_i)\not=-\alpha_i$ and $(\epsilon_{\lambda}(\sigma_i),\alpha_j)=1$. Denote by $\mathfrak{g}_{(i,j)}$ and $\mathfrak{b}_{(i,j)}$ the simple Lie algebra $\mathfrak{sl}_3$ and its Borel subalgebra generated by $\{e_i,e_j,h_i,h_j,f_i,f_j\}$ and $\{h_i,h_j,f_i,f_j\}$, respectively (where we use the letters $e_i,e_j,h_i,h_j,f_i,f_j$ for the Chevalley basis of $\mathfrak{g}$). 
Also, let $G_{(i,j)}$ and $B_{(i,j)}$ be the corresponding algebraic groups, respectively.
Since $(\epsilon_\lambda(\sigma_i),\alpha_j)=1$, by \eqref{tuika15}, we also have $(\epsilon_\lambda(\sigma_j),\alpha_i)=1$.
By \eqref{401-1}, we have
\begin{align}
(\epsilon_{\sigma_i\ast\lambda}(\sigma_i),\alpha_j)=(\epsilon_{\sigma_j\ast\lambda}(\sigma_j),\alpha_i)=0,
\end{align}
and thus, $(j,i,\sigma_i\ast\lambda)$ and $(i,j,\sigma_i\ast\lambda)$ satisfy \eqref{novelcond}.
By Lemma \ref{lemm:serre}, Lemma \ref{special}, and the discussion on the Serre relations above, $W(p,\sigma_i\ast\lambda)_Q^{\{i,j\}}$ is a $G_{(i,j)}$-submodule of the $B_{(i,j)}$-module $V_{\sqrt{p}Q+\sigma_i\ast\lambda}$. 
Set 
\begin{align}
x=|-\sqrt{p}(\alpha+\widehat{\sigma_i\ast\lambda})+\overline{\sigma_i\ast\lambda}\rangle\in W(p,\sigma_i\ast\lambda)_Q^J,
\end{align}
where $\alpha\in Q$ such that $(\alpha+\widehat{\sigma_i\ast\lambda},\alpha_s)\geq0$ for $s\in J$. 
Then $x$ is the highest weight vector of the irreducible $\mathfrak{g}_{(i,j)}$-module with highest weight $\alpha+\widehat{\sigma_i\ast\lambda}$. By the proof of Lemma \ref{newlemm3.4}, 
we have the nonzero vectors $y$ and $z$ in $W(p,\sigma_i\ast\lambda)_Q^J$ defined by
\begin{align}
y=f_i^{(\alpha+\widehat{\sigma_i\ast\lambda},\alpha_i)}x,~z=f_j^{(\sigma_i(\alpha+\widehat{\sigma_i\ast\lambda}),\alpha_j)}y.
\end{align}
Since $e_jy=0$, 
we have $y\in\mathcal{W}^j_{-\sqrt{p}(\sigma_i(\alpha+\widehat{\sigma_i\ast\lambda}))+\overline{\sigma_i\ast\lambda}}$, and hence 
\begin{align}\label{minialign}
f_iy=f_jz=0
\end{align}
by Corollary \ref{coram1} \eqref{coram1.2}.
By \eqref{alignfrom42}, we have 
\begin{align}\label{(1029)}
y=F_{i,\lambda}|-\sqrt{p}(\sigma_i(\alpha+\widehat{\sigma_i\ast\lambda})+\epsilon_{\lambda}(\sigma_i))+\bar{\lambda}\rangle.
\end{align}
Moreover, by Theorem \ref{am1} \eqref{am1:modstr}, the vector $z'$ defined by
\begin{align}
z'=&f_j^{(\sigma_i(\alpha+\widehat{\sigma_i\ast\lambda})+\epsilon_{\lambda}(\sigma_i),\alpha_j)}|-\sqrt{p}(\sigma_i(\alpha+\widehat{\sigma_i\ast\lambda})+\epsilon_{\lambda}(\sigma_i))+\bar{\lambda}\rangle\\
=&f_j^{(\sigma_i(\alpha+\widehat{\sigma_i\ast\lambda}),\alpha_j)+1}|-\sqrt{p}(\sigma_i(\alpha+\widehat{\sigma_i\ast\lambda})+\epsilon_{\lambda}(\sigma_i))+\bar{\lambda}\rangle\nonumber
\end{align}
is nonzero. On the other hand, we have
\begin{align}\label{(1030)}
F_{i,\lambda}z'&=F_{i,\lambda}f_j^{(\sigma_i(\alpha+\widehat{\sigma_i\ast\lambda}),\alpha_j)+1}|-\sqrt{p}(\sigma_i(\alpha+\widehat{\sigma_i\ast\lambda})+\epsilon_{\lambda}(\sigma_i))+\bar{\lambda}\rangle\\
&=f_j^{(\sigma_i(\alpha+\widehat{\sigma_i\ast\lambda}),\alpha_j)+1}F_{i,\lambda}|-\sqrt{p}(\sigma_i(\alpha+\widehat{\sigma_i\ast\lambda})+\epsilon_{\lambda}(\sigma_i))+\bar{\lambda}\rangle\nonumber\\
&=f_j^{(\sigma_i(\alpha+\widehat{\sigma_i\ast\lambda}),\alpha_j)+1}y=f_jz=0,\nonumber
\end{align}
where the second equality follows from \eqref{(19.1)},
the third equality follows from \eqref{(1029)},
and the last equality follows from \eqref{minialign}.
Because $(\sigma_i(\alpha+\widehat{\sigma_i\ast\lambda})+\epsilon_{\lambda}(\sigma_i),\alpha_s)\geq 0$ for $s\in J\setminus\{i\}$, we have
\begin{align}\label{1030}
|-\sqrt{p}(\sigma_i(\alpha+\widehat{\sigma_i\ast\lambda})+\epsilon_{\lambda}(\sigma_i))+\bar{\lambda}\rangle\in W(p,\lambda)_Q^{J\setminus\{i\}}.
\end{align}
Thus, by \eqref{(1030)} and \eqref{1030}, $z'\in W(p,\lambda)_Q^J$. 
On the other hand, if $W(p,\lambda)_Q^J$ is closed under the operator $e_j$, then we have 
\begin{align}
|-\sqrt{p}(\sigma_i(\alpha+\widehat{\sigma_i\ast\lambda})+\epsilon_{\lambda}(\sigma_i))+\bar{\lambda}\rangle\in W(p,\lambda)_Q^J.
\end{align}
Then by \eqref{(1029)}, we have $y=0$. It contradicts the fact that $y\not=0$.
 Thus, $W(p,\lambda)_Q^J$ is not closed under the operator $e_j$, and hence not a $P_J$-submodule of $V_{\sqrt{p}Q+\lambda}$.
\end{proof}

\subsection{Vanishing cohomologies in Theorem \ref{thm2}}\label{vanishingcohoms}
Let us recall notation in \hyperlink{$V(\mu)$}{Section \ref{subsect:notation}}.
In this subsection, we prove the vanishing of the sheaf cohomologies in Theorem \ref{thm2} along the line of the proof in \cite[Section 4.2]{FT}. For $J\subseteq\Pi$, $n\in\mathbb{Z}$, and $B$-module $V$, we write $H^n(P_J\times_BV)$ for the sheaf cohomology of the sheaf of sections of the homogeneous vector bundle $P_J\times_BV$ over $P_J/B$.

We first review well-known properties of homogeneous vector bundles (see \cite{Dem,F,Kum,L}).
Let us consider the case where $J=\{i\}$ for some $1\leq i\leq l$.
By combining \eqref{isomtrivbdle} with Lemma \ref{lemm:serre}, for $n\geq 0$, $\lambda\in\Lambda$, and $\mu\in\mathfrak{h}^\ast$, we have
\begin{align}\label{(202)}
H^n(P_i\times_BW(p,\lambda)_Q^i(\mu))\simeq H^n(P_i\times_B\mathbb{C}_\mu)\otimes W(p,\lambda)_Q^i.
\end{align}
Note that by the proof of \cite[III, Theorem 5.1]{H}, for $\mu\in P$ and $1\leq i\leq l$, we have
\begin{align}\label{dims}
\operatorname{dim}_{\mathbb{C}}H^n(P_i\times_B\mathbb{C}_{\mu})=
\begin{cases}
(\mu,\alpha_i)+1&n=0, (\mu,\alpha_i)\geq 0,\\
-(\mu,\alpha_i)-1&n=1, (\mu,\alpha_i)<0,\\
0&\text{otherwise}.
\end{cases}
\end{align}
Let $\pi\colon G/B\twoheadrightarrow G/P_i$ be the $P_i/B$-bundle. Note that $P_i/B\simeq\mathbb{P}^1$. 
For a continuous map $f:X\rightarrow Y$ between topological spaces $X$ and $Y$, we write $f_\ast$ and $R^if_\ast$ for the direct image functor and higher direct image functor between categories of sheaves of abelian groups, respectively.
The following lemma is also well-known (cf. e.g., \cite{F,L}).
\begin{lemm}\label{fung}
For $1\leq i\leq l$, let $\mu$ be an element in $P$ such that $(\mu+\rho,\alpha_i)\geq 0$. Then we have $\pi_\ast\mathcal{O}_{G/B}(\mu)\simeq R^1\pi_\ast\mathcal{O}_{G/B}(\sigma_i(\mu+\rho)-\rho)$ as $\mathcal{O}_{G/P_i}$-modules. In particular, we have $H^0(P_i\times_B\mathbb{C}_\mu)\simeq H^1(P_i\times_B\mathbb{C}_{\sigma_i(\mu+\rho)-\rho})$ as $P_i$-modules.
\end{lemm}
\begin{proof}
We reproduce the proof in \cite{F} for completeness.
When $(\mu+\rho,\alpha_i)=0$, by \eqref{dims}, we have $H^0(P_i\times_B\mathbb{C}_\mu)\simeq H^1(P_i\times_B\mathbb{C}_{\sigma_i(\mu+\rho)-\rho})\simeq 0$, and hence, $\pi_\ast\mathcal{O}_{G/B}(\mu)\simeq R^1\pi_\ast\mathcal{O}_{G/B}(\sigma_i(\mu+\rho)-\rho)\simeq 0$. 

We consider the cases where $(\mu,\alpha_i)\geq 0$.
Since $\pi\colon G/B\twoheadrightarrow G/P_i$ is a $\mathbb{P}^1$-bundle, by \cite[II, Ex. 7.10.(c)]{H}, we have $G/B\simeq\mathbb{P}(\varepsilon)$. Here $\varepsilon$ is a locally free sheaf on $G/P_i$ of rank $2$, and $\mathbb{P}(\varepsilon)$ is the projective space bundle associated with $\varepsilon$ (see \cite[II.7]{H}). 
For a ringed space $X$, denote by $\operatorname{Pic}X$ the Picard group of $X$ (see \cite[II, 6]{H}), and for a projective scheme $X$, we write $\mathcal{O}_{X}(1)$ for the twisting sheaf of $X$ and $\mathcal{O}_{X}(m):=\mathcal{O}_{X}(1)^{\otimes m}$ (see \cite[II, 5]{H}). We omit the letter $X$ when it is clear. Then by \cite[II, Ex. 7.9.(a)]{H}, we have the group isomorphism
\begin{align}\label{pic}
\operatorname{Pic}G/P_i\times\mathbb{Z}\rightarrow\operatorname{Pic}G/B,~(\mathcal{M},m)\mapsto\pi_\ast\mathcal{M}(m):=\pi^\ast\mathcal{M}\otimes_{\mathcal{O}_{G/B}}\mathcal{O}_{G/B}(m).
\end{align}
By \eqref{pic}, we have $\mathcal{O}(\mu)=\pi^\ast\mathcal{M}\otimes\mathcal{O}(m)$ for some $\mathcal{M}\in\operatorname{Pic}G/P_i$ and $m=(\mu,\alpha_i)$. By \cite[Lemme (i)]{Dem1}, the relative canonical sheaf $\omega_{(G/B)/(G/P_i)}=\Omega_{(G/B)/(G/P_i)}$ (see \cite[III, Ex. 8.4(b)]{H}) is $\mathcal{O}(-\alpha_i)$.
Then we have
\begin{align}\label{extra1}
\mathcal{O}(\sigma_i(\mu+\rho)-\rho)&\simeq\mathcal{O}(\mu)\otimes\omega_{(G/B)/(G/P_i)}^{\otimes m+1}\\
&\simeq\pi^\ast\mathcal{M}\otimes\mathcal{O}(m)\otimes(\pi^\ast\wedge^2\varepsilon\otimes\mathcal{O}(-2))^{\otimes m+1}\nonumber\\
&\simeq\pi^\ast(\mathcal{M}\otimes(\wedge^2\varepsilon)^{\otimes m+1})\otimes\mathcal{O}(-m-2),\nonumber
\end{align}
where the second $\mathcal{O}_{G/B}$-module isomorphism follows from \cite[III, Ex. 8.4(b)]{H}. Moreover, by \eqref{extra1}, we have
\begin{align}\label{extra2}
R^1\pi_\ast\mathcal{O}(\sigma_i(\mu+\rho)-\rho)&\simeq R^1\pi_\ast(\pi^\ast(\mathcal{M}\otimes(\wedge^2\varepsilon)^{\otimes m+1})\otimes\mathcal{O}(-m-2))\\
&\simeq\mathcal{M}\otimes(\wedge^2\varepsilon)^{\otimes m+1}\otimes R^1\pi_\ast\mathcal{O}(-m-2)\nonumber\\
&\simeq\mathcal{M}\otimes(\wedge^2\varepsilon)^{\otimes m+1}\otimes ((\pi_\ast\mathcal{O}(m))^{\vee}\otimes(\wedge^2\varepsilon)^{\vee})\nonumber\\
&\simeq\mathcal{M}\otimes((\pi_\ast\mathcal{O}(m))^{\vee}\otimes(\wedge^2\varepsilon)^{\otimes m}),\nonumber
\end{align}
where the second $\mathcal{O}_{G/P_i}$-module isomorphism follows from \cite[III, Ex. 8.3]{H} and the third $\mathcal{O}_{G/P_i}$-module isomorphism follows from \cite[III, Ex. 8.4.(c)]{H}.

On the other hand, we have the perfect pairing 
\begin{align}\label{extra3}
S^m(\varepsilon)\times S^m(\varepsilon)\rightarrow(\wedge^2\varepsilon)^{\otimes m},~(x^ay^b,x^cy^d)\mapsto\delta_{a+c,m}\delta_{b+d,m}(x\wedge y)^{\otimes m},
\end{align}
where $S(\varepsilon)=\bigoplus_{m\geq 0}S^m(\varepsilon)$ is the symmetric algebra of $\varepsilon$ (see \cite[II, Ex. 5.16(a)]{H}).
By \cite[III, Ex. 8.4(a)]{H} and the assumption that $m=(\mu,\alpha_i)\geq 0$, we have $\pi_\ast\mathcal{O}_{\mathbb{P}(\varepsilon)}(m)\simeq S^m(\varepsilon)$, and hence
\begin{align}\label{extra4}
\pi_\ast\mathcal{O}(m)^{\vee}\otimes(\wedge^2\varepsilon)^{\otimes m}\simeq S^m(\varepsilon)^{\vee}\otimes(\wedge^2\varepsilon)^{\otimes m}\simeq S^m(\varepsilon)\simeq\pi_\ast\mathcal{O}(m).
\end{align}
Therefore, by \eqref{extra2} and \eqref{extra4}, we obtain that
\begin{align}
R^1\pi_\ast\mathcal{O}(\sigma_i(\mu+\rho)-\rho)\simeq\mathcal{M}\otimes\pi_\ast\mathcal{O}(m)\simeq\pi_\ast\mathcal{O}(\mu),
\end{align}
where the last $\mathcal{O}_{G/P_i}$-module isomorphism follows from \cite[III, Ex. 8.3]{H}.
\end{proof}

By combining Lemma \ref{fung} with \eqref{(401)}, we obtain the following.
\begin{cor}\label{corfung}
For $1\leq i\leq l$, $\sigma\in W$, and $\lambda\in\Lambda$, we have
\begin{align*}
H^a(P_i\times_B\mathbb{C}_{\epsilon_{\lambda}(\sigma)})&\simeq
H^b(P_i\times_B\mathbb{C}_{\epsilon_{\lambda}(\sigma_i\sigma)+\epsilon_{\sigma_i\sigma\ast\lambda}(\sigma_i)+\delta_{(\sqrt{p}\overline{\sigma\ast\lambda},\alpha_i),p-1}\alpha_i})
\end{align*}
as $P_i$-modules, where $(a,b)=(0,1)$ or $(1,0)$.
\end{cor}

\begin{conj}\label{van}
Let us fix $\lambda\in\Lambda$ and $\sigma\in W$. If $l(\sigma_i\sigma)=l(\sigma)+1$, then we have 
\begin{align}\label{conjvan}
H^n(\xi_{\sigma\ast\lambda}(\epsilon_{\lambda}(\sigma)))\simeq H^{n+1}(\xi_{\sigma_i\sigma\ast\lambda}(\epsilon_{\lambda}(\sigma_i\sigma)))
\end{align}
for any $n\in\mathbb{Z}$ (see \eqref{(666)}).
In particular, for any $\lambda\in\Lambda$, $\sigma\in W$ and $n\in\mathbb{Z}$, we have $H^n(\xi_{\lambda})\simeq H^{n+l(\sigma)}(\xi_{\sigma\ast\lambda}(\epsilon_{\lambda}(\sigma)))$.
\end{conj}

We prove Conjecture \ref{van} partially (see also Remark \ref{tech} below).

\begin{thm}\label{vanish}
Let us fix $\lambda\in\Lambda$ and $\sigma\in W$. If $l(\sigma_i\sigma)=l(\sigma)+1$ and $(\epsilon_{\lambda}(\sigma),\alpha_i)=0$, then \eqref{conjvan} holds for any $n\in\mathbb{Z}$.
\end{thm}

\begin{rmk}
When $\mathfrak{g}=\mathfrak{sl}_2$, Conjecture \ref{van} is equivalent to Theorem \ref{vanish}.
\end{rmk}

We first prove a lemma needed for the proof of Theorem \ref{vanish}.

\begin{lemm}\label{beforevan1}
Let $1\leq i\leq l$, $\sigma\in W$ such that $l(\sigma\sigma_i)=l(\sigma)+1$, $\lambda\in\Lambda$ such that $(\sqrt{p}\overline{\sigma\ast\lambda},\alpha_i)\leq p-2$. Then we have short exact sequences of $P_i$-modules
\begin{align*}
0&\rightarrow H^0(P_i\times_B\mathbb{C}_{\epsilon_{\lambda}(\sigma)})\otimes W(p,\sigma\ast\lambda)_Q^i\\
&\rightarrow H^0(P_i\times_BV_{\sqrt{p}Q+\sigma\ast\lambda}(\epsilon_{\lambda}(\sigma)))\nonumber\\
&\rightarrow H^0(P_i\times_B\mathbb{C}_{\epsilon_{\lambda}(\sigma)+\epsilon_{\sigma\ast\lambda}(\sigma_i)})\otimes W(p,\sigma_i\sigma\ast\lambda)_Q^i\rightarrow 0\nonumber
\end{align*} 
and 
\begin{align*}
0&\rightarrow H^1(P_i\times_B\mathbb{C}_{\epsilon_{\lambda}(\sigma_i\sigma)})\otimes W(p,\sigma_i\sigma\ast\lambda)_Q^i\\
&\rightarrow H^1(P_i\times_BV_{\sqrt{p}Q+\sigma_i\sigma\ast\lambda}(\epsilon_{\lambda}(\sigma_i\sigma)))\nonumber\\
&\rightarrow H^1(P_i\times_B\mathbb{C}_{\epsilon_{\lambda}(\sigma_i\sigma)+\epsilon_{\sigma_i\sigma\ast\lambda}(\sigma_i)})\otimes W(p,\sigma\ast\lambda)_Q^i
\rightarrow 0.\nonumber
\end{align*}
Moreover, we have
\begin{align*}
H^1(P_i\times_BV_{\sqrt{p}Q+\sigma\ast\lambda}(\epsilon_{\lambda}(\sigma)))\simeq H^0(P_i\times_BV_{\sqrt{p}Q+\sigma_i\sigma\ast\lambda}(\epsilon_{\lambda}(\sigma_i\sigma)))\simeq 0.
\end{align*}
\end{lemm}
\begin{proof}
By Theorem \ref{am1} \eqref{am1:exseq}, we have the short exact sequence of sheaves
\begin{align}\label{shortexact1}
0&\rightarrow P_i\times_BW(p,\sigma\ast\lambda)_Q^i(\epsilon_{\lambda}(\sigma))\\
&\rightarrow P_i\times_BV_{\sqrt{p}Q+\sigma\ast\lambda}(\epsilon_{\lambda}(\sigma))\nonumber\\
&\rightarrow P_i\times_BW(p,\sigma_i\sigma\ast\lambda)_Q^i(\epsilon_{\lambda}(\sigma)+\epsilon_{\sigma\ast\lambda}(\sigma_i))\rightarrow 0.\nonumber
\end{align}
Because $(\epsilon_{\lambda}(\sigma),\alpha_i)\geq0$ (see Remark \ref{rmkbeloweps}) and $(\epsilon_{\lambda}(\sigma)+\epsilon_{\sigma\ast\lambda}(\sigma_i),\alpha_i)\geq-1$ (see Remark \ref{beloweps}), by \eqref{dims}, we have 
\begin{align}\label{align001}
H^1(P_i\times_B\mathbb{C}_{\epsilon_{\lambda}(\sigma)})\simeq H^1(P_i\times_B\mathbb{C}_{\epsilon_{\lambda}(\sigma)+\epsilon_{\sigma\ast\lambda}(\sigma_i)})\simeq 0.
\end{align}
By combining the long exact sequence corresponding to \eqref{shortexact1} with \eqref{(202)} and  \eqref{align001}, we obtain the short exact sequence of $P_i$-modules
\begin{align}\label{(400)}
0&\rightarrow H^0(P_i\times_B\mathbb{C}_{\epsilon_{\lambda}(\sigma)})\otimes W(p,\sigma\ast\lambda)_Q^i\\
&\rightarrow H^0(P_i\times_BV_{\sqrt{p}Q+\sigma\ast\lambda}(\epsilon_{\lambda}(\sigma)))\nonumber\\
&\rightarrow H^0(P_i\times_B\mathbb{C}_{\epsilon_{\lambda}(\sigma)+\epsilon_{\sigma\ast\lambda}(\sigma_i)})\otimes W(p,\sigma_i\sigma\ast\lambda)_Q^i\rightarrow 0.\nonumber
\end{align}
Here, by \eqref{align001}, we have 
\begin{align}\label{align007}
H^1(P_i\times_BV_{\sqrt{p}Q+\sigma\ast\lambda}(\epsilon_{\lambda}(\sigma)))\simeq0.
\end{align}

On the other hand, by applying the exact sequence in Theorem \ref{am1} \eqref{am1:exseq} to the case of $\sigma_i\sigma\ast\lambda$, we have the short exact sequence of sheaves
\begin{align}\label{shortexact2}
0&\rightarrow P_i\times_B W(p,\sigma_i\sigma\ast\lambda)_Q^i(\epsilon_{\lambda}(\sigma_i\sigma))\\
&\rightarrow P_i\times_BV_{\sqrt{p}Q+\sigma_i\sigma\ast\lambda}(\epsilon_{\lambda}(\sigma_i\sigma))\nonumber\\
&\rightarrow P_i\times_B W(p,\sigma\ast\lambda)_Q^i(\epsilon_\lambda(\sigma_i\sigma)+\epsilon_{\sigma_i\sigma\ast\lambda}(\sigma_i))
\rightarrow 0.\nonumber
\end{align}
By \eqref{align001} and Corollary \ref{corfung}, we have
\begin{align}\label{tuika2}
H^0(P_i\times_B\mathbb{C}_{\epsilon_\lambda(\sigma_i\sigma)})\simeq H^0(P_i\times_B\mathbb{C}_{\epsilon_\lambda(\sigma_i\sigma)+\epsilon_{\sigma_i\sigma\ast\lambda}(\sigma_i)})\simeq 0.
\end{align}
By combining the long exact sequence corresponding to \eqref{shortexact2} with \eqref{(202)} and \eqref{tuika2}, we obtain the short exact sequence of $P_i$-modules
\begin{align}\label{(402.5)}
0&\rightarrow H^1(P_i\times_B\mathbb{C}_{\epsilon_{\lambda}(\sigma_i\sigma)})\otimes W(p,\sigma_i\sigma\ast\lambda)_Q^i\\
&\rightarrow H^1(P_i\times_BV_{\sqrt{p}Q+\sigma_i\sigma\ast\lambda}(\epsilon_{\lambda}(\sigma_i\sigma)))\nonumber\\
&\rightarrow H^1(P_i\times_B\mathbb{C}_{\epsilon_{\lambda}(\sigma_i\sigma)+\epsilon_{\sigma_i\sigma\ast\lambda}(\sigma_i)})\otimes W(p,\sigma\ast\lambda)_Q^i
\rightarrow 0.\nonumber
\end{align}
Here, by \eqref{tuika2}, we have 
\begin{align}\label{align011}
H^0(P_i\times_BV_{\sqrt{p}Q+\sigma_i\sigma\ast\lambda}(\epsilon_{\lambda}(\sigma_i\sigma)))\simeq 0.
\end{align}
Thus, Lemma \ref{beforevan1} is proved.
\end{proof}

\begin{proof}[Proof of Theorem \ref{vanish}]
We first consider the cases where $(\sqrt{p}\overline{\sigma\ast\lambda},\alpha_i)\leq p-2$. 
By Remark \ref{beloweps} and the assumption that $(\epsilon_{\lambda}(\sigma),\alpha_i)=0$, we have
\begin{align}\label{tuika1}
(\epsilon_{\lambda}(\sigma)+\epsilon_{\sigma\ast\lambda}(\sigma_i),\alpha_i)=-1.
\end{align}
By \eqref{dims}, \eqref{tuika1}, and Corollary \ref{corfung}, we have 
\begin{align}\label{tuika3}
H^0(P_i\times_B\mathbb{C}_{\epsilon_{\lambda}(\sigma)+\epsilon_{\sigma\ast\lambda}(\sigma_i)})\simeq H^1(P_i\times_B\mathbb{C}_{\epsilon_\lambda(\sigma_i\sigma)})\simeq0.
\end{align}
Then we obtain the $P_i$-module isomorphisms
\begin{align}\label{(403)}
&H^0(P_i\times_BV_{\sqrt{p}Q+\sigma\ast\lambda}(\epsilon_{\lambda}(\sigma)))\\
\simeq~&H^0(P_i\times_B\mathbb{C}_{\epsilon_{\lambda}(\sigma)})\otimes W(p,\sigma\ast\lambda)_Q^i\nonumber\\
\simeq~&H^1(P_i\times_B\mathbb{C}_{\epsilon_{\lambda}(\sigma_i\sigma)+\epsilon_{\sigma_i\sigma\ast\lambda}(\sigma_i)})\otimes W(p,\sigma\ast\lambda)_Q^i\nonumber\\
\simeq~&H^1(P_i\times_BV_{\sqrt{p}Q+\sigma_i\sigma\ast\lambda}(\epsilon_{\lambda}(\sigma_i\sigma))),\nonumber
\end{align}
where the first and third isomorphisms follow from \eqref{tuika3} and Lemma \ref{beforevan1}, and the second isomorphism follows from Corollary \ref{corfung}.

We apply the Leray spectral sequences
\begin{align}
\begin{cases}
E_2^{a,b}=H^a(G\times_{P_i}H^b(P_i\times_BV_{\sqrt{p}Q+\sigma\ast\lambda}(\epsilon_{\lambda}(\sigma))))\Rightarrow H^{a+b}(\xi_{\sigma\ast\lambda}(\epsilon_{\lambda}(\sigma)))\\
E_2^{a,b}=H^a(G\times_{P_i}H^b(P_i\times_BV_{\sqrt{p}Q+\sigma_i\sigma\ast\lambda}(\epsilon_{\lambda}(\sigma_i\sigma))))\Rightarrow H^{a+b}(\xi_{\sigma_i\sigma\ast\lambda}(\epsilon_{\lambda}(\sigma_i\sigma)))\nonumber
\end{cases}
\end{align}
to the fibration $\pi:G/B\rightarrow G/P_i$ and the vector bundle $\xi_{\lambda}$ over $G/B$. We assume that $a+b=n$ for the first spectral sequence, and that $a+b=n+1$ for the second one above. Then by \eqref{(403)} and Lemma \ref{beforevan1}, we have
\begin{align}
H^n(\xi_{\sigma\ast\lambda}(\epsilon_{\lambda}(\sigma)))\simeq H^{n+1}(\xi_{\sigma_i\sigma\ast\lambda}(\epsilon_{\lambda}(\sigma_i\sigma)))
\end{align}
as $G$-modules, because both hands have the same spectral sequence.

Second, we consider the case of $(\sqrt{p}\overline{\sigma\ast\lambda},\alpha_i)=p-1$. By Lemma \ref{lemm:serre} and the fact that $V_{\sqrt{p}Q+\sigma\ast\lambda}=W(p,\sigma\ast\lambda)_Q^i$, $V_{\sqrt{p}Q+\sigma\ast\lambda}$ is a $P_i$-module. 
By \eqref{401-1}, \eqref{(401)}, and the assumption that $(\epsilon_\lambda(\sigma),\alpha_i)=0$, we have
\begin{align}\label{94}
\epsilon_{\lambda}(\sigma_i\sigma)=\sigma_i(\epsilon_{\lambda}(\sigma)+\rho)-\rho.
\end{align}
By Lemma \ref{fung} and \eqref{94}, we obtain a $P_i$-module isomorphisms
\begin{align}\label{align012}
H^a(P_i\times_BV_{\sqrt{p}Q+\sigma\ast\lambda}(\epsilon_{\lambda}(\sigma)))&\simeq H^a(P_i\times_B\mathbb{C}_{\epsilon_{\lambda}(\sigma)})\otimes V_{\sqrt{p}Q+\sigma\ast\lambda}\\
&\simeq H^{b}(P_i\times_B\mathbb{C}_{\epsilon_{\lambda}(\sigma_i\sigma)})\otimes V_{\sqrt{p}Q+\sigma_i\sigma\ast\lambda}\nonumber\\
&\simeq H^{b}(P_i\times_BV_{\sqrt{p}Q+\sigma_i\sigma\ast\lambda}(\sigma_i(\epsilon_{\lambda}(\sigma_i\sigma))))\nonumber
\end{align}
for $(a,b)=(0,1)$ or $(1,0)$.
On the other hand, by \eqref{94}, Lemma \ref{fung}, and the fact that $(\epsilon_\lambda(\sigma),\alpha_i)\geq 0$ (see Remark \ref{rmkbeloweps}), we have
\begin{align}\label{align0125}
H^1(P_i\times_B\mathbb{C}_{\epsilon_\lambda(\sigma)})\simeq H^0(P_i\times_B\mathbb{C}_{\epsilon_{\lambda}(\sigma_i\sigma)})\simeq 0, 
\end{align}
and thus, by \eqref{align012} and \eqref{align0125}, we have 
\begin{align}\label{align013}
H^1(P_i\times_BV_{\sqrt{p}Q+\sigma\ast\lambda}(\epsilon_{\lambda}(\sigma)))\simeq H^0(P_i\times_BV_{\sqrt{p}Q+\sigma_i\sigma\ast\lambda}(\epsilon_{\lambda}(\sigma_i\sigma)))\simeq 0.
\end{align}
By combining the Leray spectral sequence above with \eqref{align012} and \eqref{align013}, we have $H^n(\xi_{\sigma\ast\lambda}(\epsilon_{\lambda}(\sigma)))\simeq H^{n+1}(\xi_{\sigma_i\sigma\ast\lambda}(\epsilon_{\lambda}(\sigma_i\sigma)))$.
\end{proof}

\begin{cor}\label{restrict}
For $n\geq 1$ and $\lambda\in\Lambda$ such that $(\sqrt{p}\bar\lambda+\rho,\theta)\leq p$, we have $H^n(\xi_{\lambda})=0$. 
\end{cor}
\begin{proof}
Let us recall Lemma \ref{condequiv}.
For a minimal expression $\sigma_{i_{l(w_0)}}\cdots\sigma_{i_1}$ of $w_0$,
applying Theorem \ref{vanish} to $\sigma_{i_n}\cdots\sigma_{i_1}$ for $1\leq n\leq l(w_0)$ repeatedly, we obtain that
$H^n(\xi_{\lambda})\simeq H^{n+l(w_0)}(\xi_{w_0\ast\lambda}(\epsilon_{\lambda}(w_0)))$.
Since $l(w_0)=\dim G/B$,  Corollary \ref{restrict} is proved.
\end{proof}

\begin{rmk}\label{tech}
If $(\epsilon_{\lambda}(\sigma),\alpha_i)>0$, then \eqref{tuika3} does not hold.
In particular, the first and third isomorphisms in \eqref{(403)} do not hold.
\end{rmk}

\subsection{Character formulas in Theorem \ref{thm2}}\label{chformulas}
We give the character formulas in Theorem \ref{thm2} along the line of the proof in \cite[Section 6]{FT}. For $\alpha\in P$, we use the notation
$z^{\alpha}=z_1^{(\alpha_1,\alpha)}\dots z_l^{(\alpha_l,\alpha)}$.
Let $V$ be a vector space equipped with diagonalizable actions of $L_0$ and $\mathfrak{h}$. Then the {\it full character} $\chi_V(q, z)$ of $V$ is defined by
$\chi_V(q, z)=\operatorname{tr}_V(q^{L_0-\frac{c}{24}}z_1^{h_1}\cdots z_l^{h_l})$.
We consider the case where $V={H^0(\xi_{\lambda})}$ for $\lambda\in\Lambda$ such that $(\sqrt{p}\bar\lambda+\rho,\theta)\leq p$.

Let us recall the open covering \eqref{(33.001)} of $G/B$, and regard $H^0(\xi_{\lambda})$ as the subspace of $\bigoplus_{\sigma\in W}H^0(\xi_{\lambda}|_{U_{\sigma}})=\bigoplus_{\sigma\in W}\mathcal{O}_{G/B}(U_{\sigma})\otimes V_{\sqrt{p}Q+\lambda}$ by the embedding \eqref{H0embedding}. Then for $v\in\mathcal{F}_{-\sqrt{p}(\alpha+\hat\lambda)+\bar\lambda}$, the $G$-action defined by \eqref{(31.2)} induces the following action of $\mathfrak{h}$ on $\mathcal{O}_{G/B}(U_{\sigma})\otimes\mathcal{F}_{-\sqrt{p}\alpha+\lambda}\subseteq\mathcal{O}_{G/B}(U_{\sigma})\otimes V_{\sqrt{p}Q+\lambda}$:
\begin{align}\label{(59)}
&h_i({x_{\beta_1,\sigma}}^{n_1}\cdots{x_{\beta_m,\sigma}}^{n_m}\otimes v)\\
=&(\sigma(\alpha+\hat\lambda+\rho)-\rho-\sum_{j=1}^{m}n_m\beta_m,\alpha_i){x_{\beta_1,\sigma}}^{n_1}\cdots{x_{\beta_m,\sigma}}^{n_m}\otimes v\nonumber.
\end{align}
Here ${x_{\beta_1,\sigma}},\dots,{x_{\beta_m,\sigma}}$ are generators of $\mathcal{O}_{G/B}(U_{\sigma})$ corresponding to the positive roots $\beta_1,\cdots,\beta_m$. 
The action of $L_0$ on $\mathcal{O}_{G/B}(U_{\sigma})\otimes V_{\sqrt{p}Q+\lambda}$ is given by
\begin{equation}\label{(60)}
L_0({x_{\beta_1,\sigma}}^{n_1}\cdots{x_{\beta_m,\sigma}}^{n_m}\otimes v)={x_{\beta_1,\sigma}}^{n_1}\cdots{x_{\beta_m,\sigma}}^{n_m}\otimes L_0v
\end{equation}
by definition.
By Remark \ref{holom}, we have
$H^n(\xi_{\lambda})=\bigoplus_{\Delta\in\mathbb{Q}}H^n(({\xi_{\lambda}})_{\Delta})$
for $n\geq 0$. Using the Atiyah-Bott fixed point formula \cite{AB} and Corollary \ref{restrict}, we obtain that
\begin{align}
\chi_{H^0(\xi_{\lambda})}(q, z)&=\sum_{n\geq 0}(-1)^n\chi_{H^n(\xi_{\lambda})}(q, z)\\
&=\sum_{\Delta\in\mathbb{Q}}q^{\Delta-\frac{c}{24}}\sum_{n\geq 0}(-1)^n\operatorname{tr}_{H^n((\xi_\lambda)_\Delta)}(z_1^{h_1}\cdots z_l^{h_l})\nonumber\\
&=\sum_{\Delta\in\mathbb{Q}}q^{\Delta-\frac{c}{24}}\sum_{\sigma\in W}\operatorname{tr}_{\mathcal{O}_{G/B}(U_{\sigma})\otimes (V_{\sqrt{p}Q+\lambda})_\Delta}(z_1^{h_1}\cdots z_l^{h_l})\nonumber\\
&=\sum_{\sigma\in W}\chi_{\mathcal{O}_{G/B}(U_{\sigma})\otimes V_{\sqrt{p}Q+\lambda}}(q,z).\nonumber
\end{align}
By \eqref{(59)}, \eqref{(60)} and calculations in \cite[Section 4.2]{BM}, we have
\begin{align}\label{(61)}
\chi_{H^0(\xi_{\lambda})}(q, z)&=\frac{1}{\eta(q)^l}\sum_{\sigma\in W}\sum_{\alpha\in Q}(-1)^{l(\sigma)}\frac{q^{\frac{1}{2}|\sqrt{p}(\alpha+\hat{\lambda}+\rho)-\bar{\lambda}-\frac{1}{\sqrt{p}}\rho|^2}z^{\sigma(\alpha+\rho+\hat{\lambda})-\rho}}{\Pi_{\beta\in\Delta_-}(1-z^{\beta})}\\
&=\sum_{\alpha\in P_+\cap Q}\chi_{\alpha+\hat{\lambda}}^{\mathfrak{g}}(z)\sum_{\sigma\in W}(-1)^{l(\sigma)}\frac{q^{\frac{1}{2}|\sqrt{p}\sigma(\alpha+\rho+\hat{\lambda})-\bar{\lambda}-\frac{1}{\sqrt{p}}\rho|^2}}{\eta(q)^l}\nonumber\\
&=\sum_{\alpha\in P_+\cap Q}\chi_{\alpha+\hat{\lambda}}^{\mathfrak{g}}(z)\operatorname{tr}_{T^p_{\sqrt{p}\bar{\lambda},\alpha+\hat{\lambda}}}(q^{L_0-\frac{c}{24}}).\nonumber
\end{align}
\hypertarget{lastsection4}
{Here $\chi_{\beta}^{\mathfrak{g}}(z)$ is the Weyl character of the irreducible $\mathfrak{g}$-module $\mathcal{R}_{\beta}$ with the highest weight $\beta$, $\eta(q)$ is the Dedekind Eta function and $T^p_{\sqrt{p}\bar{\lambda},\alpha+\hat{\lambda}}=H^0_{DS,\alpha+\hat{\lambda}}(\mathbb{V}_{p,\sqrt{p}\bar{\lambda}})$ is the $\mathcal{W}^{p-h}(\mathfrak{g})$-module defined in \cite{ArF}.}
\begin{rmk}
If Conjecture \ref{van} holds, then we have $H^n(\xi_\lambda)=0$ for any $\lambda\in\Lambda$ and $n\geq1$, and thus, \eqref{(61)} holds for any $\lambda\in\Lambda$.
\end{rmk}

\subsection{Proofs of Main Theorem and Theorem \ref{thm1}}\label{lastsection}
Let us recall Definition \ref{logWalgs}, \ref{dfnmathcalWs}.
For $\alpha\in P_+\cap Q$ and $\lambda\in \Lambda$, we use the letters $y_{\alpha+\hat\lambda}$ and  $y'_{\alpha+\hat\lambda}$ for the highest weight vector and lowest weight vector of $\mathcal{R}_{\alpha+\hat\lambda}$ throughout this subsection.
For $\lambda\in\Lambda$, we consider the $B$-module homomorphism $\Phi_\lambda$ defined by 
\begin{align}\label{embinnewthm1}
\Phi_\lambda\colon H^0(\xi_\lambda)\rightarrow V_{\sqrt{p}Q+\lambda},~s\mapsto s(\operatorname{id}B).
\end{align}
\begin{thm}\label{newthm1}
The $B$-module homomorphism $\Phi_\lambda$ is injective.
\end{thm}
We first prove some lemmas needed for the proof of Theorem \ref{newthm1}.
For a $G$-module $M$ and $\gamma\in\mathfrak{h}^\ast$, denote by $M_{\gamma}$ the weight space of $M$ with the weight $\gamma$.

\begin{lemm}\label{newlemm2}
We have 
$H^0(\xi_\lambda)\simeq\bigoplus\limits_{\alpha\in P_+\cap Q}\mathcal{R}_{\alpha+\hat\lambda}\otimes H^0(\xi_\lambda)_{\alpha+\hat\lambda}^{\mathfrak{n}_+}$
as $G$-modules.
\end{lemm}
\begin{proof}
 For any $\Delta\in\mathbb{Q}$, since $\dim(V_{\sqrt{p}Q+\lambda})_\Delta<\infty$, we have $\dim_{\mathbb{C}} H^0((\xi_\lambda)_\Delta)<\infty$. By Weyl's theorem on complete reducibility and by Remark \ref{holom}, $H^0(\xi_\lambda)$ decomposes into a direct sum of finite-dimensional irreducible $\mathfrak{g}$-modules. 
\end{proof}
For $s\in H^0(\xi_\lambda)_{\alpha+\hat\lambda}^{\mathfrak{n}_+}$, denote by $\psi_s$ the $G$-module isomorphism from $\mathcal{U}(\mathfrak{g})s$ to $\mathcal{R}_{\alpha+\hat\lambda}$.
Let us recall the open covering \eqref{(33.001)} and the $G$-module isomorphism \eqref{strofH0}.
\begin{lemm}\label{newlemm3}
If $s\in H^0(\xi_\lambda)_{\alpha+\hat\lambda}^{\mathfrak{n}_+}$, then $s|_{U_{\operatorname{id}}}=1\otimes v$ for some $v\in\mathcal{F}_{-\sqrt{p}\alpha+\lambda}$. Moreover, the linear map $\Phi_\lambda|_{H^0(\xi_\lambda)_{\alpha+\hat\lambda}^{\mathfrak{n}_+}}\colon H^0(\xi_\lambda)_{\alpha+\hat\lambda}^{\mathfrak{n}_+}\rightarrow\mathcal{F}_{-\sqrt{p}\alpha+\lambda}$ is injective.
\end{lemm}
\begin{proof}

Let $s\in H^0(\xi_\lambda)_{\alpha+\hat\lambda}^{\mathfrak{n}_+}$. 
Because the $G$-action on $H^0(\xi_\lambda)$ is given by \eqref{(31.2)}, for any $g\in N_+$, we have
$s(\operatorname{id}B)=s(gB)$.
Thus, we have $s|_{U_{\operatorname{id}}}=1\otimes v$ for some $v\in V_{\sqrt{p}Q+\lambda}$.
By \eqref{(59)}, we have $v\in\mathcal{F}_{-\sqrt{p}\alpha+\lambda}$.
For $s\in H^0(\xi_\lambda)_{\alpha+\hat\lambda}$, if $s|_{U_{\operatorname{id}}}=0$, then by \eqref{strofH0}, we have $s=0$.
\end{proof}

\begin{lemm}\label{newlemm4}
For $\Delta\in\mathbb{Q}$, $\alpha\in P_+\cap Q$ and $s\in H^0((\xi_\lambda)_\Delta)_{\alpha+\hat\lambda}^{\mathfrak{n}_+}$, the $B$-module homomorphism $\Phi_\lambda|_{\mathcal{U}(\mathfrak{g})s}\colon\mathcal{U}(\mathfrak{g})s\rightarrow\mathcal{U}(\mathfrak{b})\Phi_\lambda(s)$
is an isomorphism.
\end{lemm}
\begin{proof}
If $\Phi_\lambda(s)=0$, then by Lemma \ref{newlemm3}, we have $s=0$ and the claim is clear.
Let $\Phi_\lambda(s)\not=0$.
By Lemma \ref{newlemm3.9}, the $B$-module homomorphism $\Phi_\lambda|_{\mathcal{U}(\mathfrak{g})s}\circ \psi_s^{-1}\colon \mathcal{R}_{\alpha+\hat\lambda}\rightarrow\mathcal{U}(\mathfrak{b})\Phi_\lambda(s)$ is an isomorphism. Since $\psi_s$ is an isomorphism, so is $\Phi_\lambda|_{\mathcal{U}(\mathfrak{g})s}$.
\end{proof}

Let us recall the standard partial order $\geq$ on $\mathfrak{h}^\ast$ defined by
\begin{align}
\mu\geq\mu'\Leftrightarrow\mu-\mu'\in\sum_{i=1}^l\mathbb{Z}_{\geq0}\alpha_i.
\end{align}
\begin{lemm}\label{triviallemm}
For $\beta_1,\beta_2\in P_+$, if $\beta_2\not\geq\beta_1$, then we have $(\mathcal{R}_{\beta_2})_{w_0(\beta_1)}=\{0\}$.
\end{lemm}

\begin{proof}[Proof of Theorem \ref{newthm1}]
For $\Delta\in\mathbb{Q}$ and $\alpha\in P_+\cap Q$, let us choose a basis $(B_{-\sqrt{p}\alpha+\lambda})_{\Delta}$ of $H^0((\xi_\lambda)_\Delta)_{\alpha+\hat\lambda}^{\mathfrak{n}_+}$.
By Lemma \ref{newlemm2}, for a vector $s\in H^0(\xi_\lambda)$, we have 
\begin{align}\label{(66.3)}
s=\sum_{s^j_{-\sqrt{p}\beta+\lambda}\in K}\psi_{s^j_{-\sqrt{p}\beta+\lambda}}^{-1}(y^j_{\beta+\hat\lambda}), 
\end{align}
where $y^j_{\beta+\lambda}\in\mathcal{R}_{\beta+\hat\lambda}$ and $K$ is a finite subset of $\bigsqcup_{\Delta\in\mathbb{Q}}\bigsqcup_{\beta\in P_+\cap Q}(B_{-\sqrt{p}\beta+\lambda})_{\Delta}$. We show that if $\Phi_\lambda(s)=0$, then $s=0$.

We can assume that all $s^j_{-\sqrt{p}\beta+\lambda}$ in \eqref{(66.3)} have the same conformal weight $\Delta\in\mathbb{Q}$, and that all $y^j_{\beta+\hat\lambda}$ have the same weight $\gamma$ for some $\gamma\in Q+\hat\lambda$. 
We consider the finite set
\begin{align}
K_B=\{\beta\in P_+\cap Q~|~(B_{-\sqrt{p}\beta+\lambda})_{\Delta}\cap K\not=\phi\},
\end{align}
and let us choose an element $\alpha\in K_B$ such that $\beta\not\geq\alpha$ for any $\beta\in K_B$, $\beta\not=\alpha$.

Let us assume that $\Phi_\lambda(\psi_{s^i_{-\sqrt{p}\alpha+\lambda}}^{-1}(y^i_{\alpha+\hat\lambda}))\not=0$ for some $s^i_{-\sqrt{p}\alpha+\lambda}\in K$.
Let $f_{i_1},\dots,f_{i_m}$ be a sequence of the Chevalley gererators of $\mathfrak{n}_-$ such that $f_{i_1}\cdots f_{i_m}y^i_{\alpha+\hat\lambda}=dy'_{\alpha+\hat\lambda}$ for some $d\in\mathbb{C}^\times$.
By the assumption that all $y^j_{\beta+\hat\lambda}$ has weight $\gamma$, we have $\gamma-\sum_{j=1}^m\alpha_{i_j}=w_0(\alpha+\hat\lambda)$. Then by Lemma \ref{triviallemm}, we have
\begin{align}\label{alignnew-1}
f_{i_1}\cdots f_{i_m}y^j_{\beta+\hat\lambda}\in(\mathcal{R}_{\beta+\hat\lambda})_{w_0(\alpha+\hat\lambda)}=\{0\}
\end{align}
for $\beta\in K_B$ such that $\beta\not=\alpha$.
On the other hand, we have $f_{i_1}\cdots f_{i_m}y^j_{\alpha+\hat\lambda}=b_jy'_{\alpha+\hat\lambda}$ for some $b_j\in\mathbb{C}$.
Then by \eqref{alignnew-1} we have
\begin{align}\label{alignnew}
0=f_{i_1}\cdots f_{i_m}\Phi_\lambda(s)=&\Phi_\lambda(\sum_{s^j_{-\sqrt{p}\beta+\lambda}\in K}\psi_{s^j_{-\sqrt{p}\beta+\lambda}}^{-1}(f_{i_1}\cdots f_{i_m}y^j_{\beta+\hat\lambda}))\\
=&\Phi_\lambda(\psi_{s_{-\sqrt{p}\alpha+\lambda}}^{-1}(y'_{\alpha+\hat\lambda})),\nonumber
\end{align}
where we have put
\begin{align}
s_{-\sqrt{p}\alpha+\lambda}:=ds^i_{-\sqrt{p}\alpha+\lambda}+{\sum}_{j\not=i}b_{j}s^j_{-\sqrt{p}\alpha+\lambda}.
\end{align}
By Lemma \ref{newlemm4} and \eqref{alignnew}, we have $\psi_{s_{-\sqrt{p}\alpha+\lambda}}^{-1}(y'_{\alpha+\hat\lambda})=0$. Thus, we have $s_{-\sqrt{p}\alpha+\lambda}=0$.
However, it contradicts the fact that the vectors in $(B_{\alpha+\hat\lambda})_\Delta$ are linearly independent. 
Thus, we obtain that $\Phi_\lambda(\psi_{s^i_{-\sqrt{p}\alpha+\lambda}}^{-1}(y^i_{\alpha+\hat\lambda}))=0$, and by Lemma \ref{newlemm4}, $\psi_{s^i_{-\sqrt{p}\alpha+\lambda}}^{-1}(y^i_{\alpha+\hat\lambda})=0$.
By repeating this process finitely many times, we obtain that $\psi_{s^j_{-\sqrt{p}\beta+\lambda}}^{-1}(y^j_{\beta+\hat\lambda})=0$ for any $s^j_{-\sqrt{p}\beta+\lambda}\in K$. Therefore, we have $s=0$ and Theorem \ref{newthm1} is proved.
\end{proof}

\begin{lemm}\label{embedding in intersection}
For any $\lambda$, $\Phi_\lambda(H^0(\xi_\lambda))$ is a $G$-submodule of $W(p,\lambda)_Q$.
\end{lemm}
\begin{proof}
By the first isomorphism in \eqref{(403)} for $\sigma=\operatorname{id}$, we have the $P_i$-module isomorphism
\begin{align}\label{tuika21}
\Phi_\lambda^i\colon H^0(P_i\times_BV_{\sqrt{p}Q+\lambda})\simeq W(p,\lambda)_Q^i,~s\mapsto s(\operatorname{id}B).
\end{align}
By definition \eqref{embinnewthm1}, $\Phi_\lambda(H^0(\xi_\lambda))$ is a $G$-submodule of $V_{\sqrt{p}Q+\lambda}$, and hence a $P_i$-submodule of $V_{\sqrt{p}Q+\lambda}$ for any $1\leq i\leq l$.
By Lemma \ref{lemm:maximalP_Jsubmod}, $\Phi_\lambda(H^0(\xi_\lambda))$ is a $G$-submodule of $W(p,\lambda)_Q^i$ for $1\leq i\leq l$, and hence a $G$-submodule of $W(p,\lambda)_Q=\bigcap_{i=1}^lW(p,\lambda)_Q^i$.
\end{proof}

\begin{proof}[Proof of Main Theorem \eqref{mthm1}, \eqref{mthm2}]
By Theorem \ref{parab}, if and only if $\lambda$ satisfies \eqref{novelcond3}, 
then $W(p,\lambda)_Q$ is a $G$-submodule of $V_{\sqrt{p}Q+\lambda}$, and hence a $G$-submodule of $\Phi_\lambda(H^0(\xi_\lambda))$ by Lemma \ref{lemm:maximalP_Jsubmod}.
Thus, by Lemma \ref{embedding in intersection}, we obtain that
$W(p,\lambda)_Q=\Phi_\lambda(H^0(\xi_\lambda))$ if and only if $\lambda$ satisfies \eqref{novelcond3}.
In particular, since $\lambda=0$ satisfies \eqref{novelcond3}, we have $H^0(\xi_0)\simeq W(p)_Q$ as $G$-modules. 
It follows that the map 
$\Phi_\lambda: H^0(\xi_\lambda)\xrightarrow{\sim}W(p,\lambda)_Q $ 
is an $W(p)_Q$-module isomorphism.
Finally, the claim $G\subseteq\operatorname{Aut}_{W(p)_Q}(H^0(\xi_\lambda))$ follows from \hyperlink{lastcomment}{Section \ref{subsectvectbdle}}.
Therefore, Main Theorem \eqref{mthm1}, \eqref{mthm2} are proved.
\end{proof}

\begin{lemm}\label{bmodgmod}
For any nonzero vector $x\in\mathcal{W}_{-\sqrt{p}\alpha+\lambda}$, we have $\mathcal{U}(\mathfrak{b})x\simeq\mathcal{R}_{\alpha+\hat\lambda}$.
\end{lemm}
\begin{proof}
By Definition \ref{dfnmathcalWs} and Corollary \ref{coram1} \eqref{coram1.2}, we have $h_{i,\lambda}x=(\alpha+\hat\lambda,\alpha_i)x$ and $f_i^{(\alpha+\hat\lambda,\alpha_i)+1}x=0$ for any $x\in\mathcal{W}_{-\sqrt{p}\alpha+\lambda}$.
Therefore, there is a well-defined $B$-module homomorphism
$\mathcal{R}_{\alpha+\hat\lambda}\rightarrow U(\mathfrak{b})x$.
The assertion follows from Lemma \ref{newlemm3.9}.
\end{proof}

\begin{proof}[Proof of Main Theorem \eqref{mthm3}]
It is enough to show that $\Phi_\lambda(H^0(\xi_\lambda)_{\alpha+\hat\lambda}^{\mathfrak{n}_+})=\mathcal{W}_{-\sqrt{p}\alpha+\lambda}$ for any $\alpha\in P_+\cap Q$ and $\lambda\in\Lambda$.
We first show that 
$\Phi_\lambda(H^0(\xi_\lambda)_{\alpha+\hat\lambda}^{\mathfrak{n}_+})\subseteq\mathcal{W}_{-\sqrt{p}\alpha+\lambda}$.
Because $\mathcal{W}^i_{-\sqrt{p}\alpha+\lambda}$ is the space of highest weight vectors of the $P_i$-submodule $W(p,\lambda)_Q^i$ of $V_{\sqrt{p}Q+\lambda}$ with the highest weight $\alpha+\hat{\lambda}$,  we have
\begin{align}\label{(1024)}
\Phi_\lambda(H^0(\xi_\lambda)_{\alpha+\hat\lambda}^{\mathfrak{n}_+})\subseteq\bigcap_{i=1}^l\mathcal{W}^i_{-\sqrt{p}\alpha+\lambda}=\mathcal{W}_{-\sqrt{p}\alpha+\lambda}.
\end{align}
We prove the converse.
Let $x\in\mathcal{W}_{-\sqrt{p}\alpha+\lambda}$.
By Lemma \ref{bmodgmod}, $\mathcal{U}(\mathfrak{b})x$ is a $G$-submodule of $V_{\sqrt{p}Q+\lambda}$.
Then by Lemma \ref{lemm:maximalP_Jsubmod}, 
$\mathcal{U}(\mathfrak{b})x$ is a $G$-submodule of $\Phi_\lambda(H^0(\xi_\lambda))$.
Since $e_ix=0$ and $h_{i,\lambda}x=(\alpha+\hat\lambda,\alpha_i)x$ for any $1\leq i\leq l$ by definition, we have
\begin{align}
x\in(\Phi_\lambda(H^0(\xi_\lambda)))^{\mathfrak{n}_+}_{\alpha+\hat\lambda}=\Phi_\lambda(H^0(\xi_\lambda)^{\mathfrak{n}_+}_{\alpha+\hat\lambda}),
\end{align}
and thus, $\mathcal{W}_{-\sqrt{p}\alpha+\lambda}\subseteq\Phi_\lambda(H^0(\xi_\lambda)^{\mathfrak{n}_+}_{\alpha+\hat\lambda})$. Thus, Main Theorem \eqref{mthm3} is proved.
\end{proof}

Let us recall that the Coxeter number $h$ and the dual Coxeter number $h^\vee$ coincides because $\mathfrak{g}$ is simply-laced. Let $k=p-h$.

\begin{proof}[Proof of Theorem \ref{thm1}]
Let us assume that $p\geq h-1$.
By Theorem \ref{thm2} and Main Theorem \eqref{mthm3}, the character of $\mathcal{W}_0$ coincides with the character of the $W$-algebra $\mathcal{W}^k(\mathfrak{g})$ (\cite{FB,FF}). By \cite[Remark 4.1]{FKRW}, we have 
$\mathcal{W}^k(\mathfrak{g})\subseteq\mathcal{W}_0$,
and thus, we obtain that $\mathcal{W}^k(\mathfrak{g})\simeq\mathcal{W}_0$.
\end{proof}

\begin{lemm}[{\cite[Theorem 10]{CrM}}] \label{lemm:crm}
When $p\geq h$, we have $\mathcal{W}^{p-h}(\mathfrak{g})\simeq\mathcal{W}_{p-h}(\mathfrak{g})$, where $\mathcal{W}_{p-h}(\mathfrak{g})$ is the simple quotient of $\mathcal{W}^{p-h}(\mathfrak{g})$.
\end{lemm}
\begin{proof}
We provide a proof for completeness. 
By the Feigin-Frenkel duality \cite{FF1,FF2,ACL}, for $k'=\frac{1}{p}-h$, we have $\mathcal{W}^k(\mathfrak{g})\simeq\mathcal{W}^{k'}(\mathfrak{g})$.
Thus, it is enough to show that $\mathcal{W}^{k'}(\mathfrak{g})$ is simple.
By \cite[Remark 6.4.1]{Ar}, we have $\mathcal{W}^{k'}(\mathfrak{g})=H^0_+(V_{k'}(\mathfrak{g}))$. Here $V_{k'}(\mathfrak{g})$ is the universal affine vertex algebra of level $k'$ associated with $\mathfrak{g}$, and $H^0_+$ is the functor in \cite[(289)]{Ar}. By \cite{GK}, $V_{k'}(\mathfrak{g})$ is the simple affine vertex algebra $L(k'\Lambda_0)$ of level $k'$, and thus, $H^0_+(V_{k'}(\mathfrak{g}))=H^0_+(L(k'\Lambda_0))$. Now, we assume that $p\geq h$, that is equivalent to the condition \cite[(383)]{Ar} in our case, to use \cite[Theorem 9.1.4]{Ar}. Under the condition, we obtain that $H^0_+(L(k'\Lambda_0))=\mathbb{L}(\gamma_{-(k'+h)\rho})$ and the right-hand side is the unique simple quotient $\mathcal{W}_{k'}(\mathfrak{g})$ (see \cite[(277)]{Ar}). Thus, Lemma \ref{lemm:crm} is proved.
\end{proof}

The followings are implicitly conjectured in \cite{FT}.

\begin{conj}\label{future}
Let $\mathfrak{g}$ be a simply-laced simple Lie algebra, $p\in\mathbb{Z}_{\geq2}$, and $\lambda\in\Lambda$.
\begin{enumerate}
\item\label{futureconj1}
The $W(p)_Q$-module $H^0(\xi_\lambda)$ is irreducible.
\item\label{futureconj2}
For any $\alpha\in P_+\cap Q$, $\mathcal{W}_{-\sqrt{p}\alpha+\lambda}$ is an irreducible $\mathcal{W}_{p-h}(\mathfrak{g})$-module.
\item \label{futureconj3}
The set $\{H^0(\xi_\lambda)|\lambda\in\Lambda\}$ provides a complete set
of isomorphism classes of  irreducible $W(p)_Q$-modules.
\end{enumerate}
\end{conj}

\begin{rmk}
In \cite{AM1}, Conjecture \ref{future} has been proved in the case of $\mathfrak{g}=\mathfrak{sl}_2$.
Recently, in \cite{AMW},
Conjecture \ref{future} \eqref{futureconj1}, \eqref{futureconj2} has been proved in the case of $\mathfrak{g}=\mathfrak{sl}_3$, $p=2$, $\lambda=0$. In the subsequent paper \cite{S1}, we will prove Conjecture \ref{future} \eqref{futureconj1}, \eqref{futureconj2} for any simply-laced simple Lie algebra $\mathfrak{g}$ and $\lambda\in\Lambda$ such that $(\sqrt{p}\bar\lambda+\rho,\theta)\leq p$.
\end{rmk}

\begin{rmk}
At least for $\lambda\in\Lambda$ such that $(\sqrt{p}\bar\lambda+\rho,\theta)\leq p$, we have
\begin{align}
\operatorname{tr}_{\mathcal{W}_{-\sqrt{p}\alpha+\lambda}}(q^{L_0-\frac{c}{24}})&=\sum_{\sigma\in W}(-1)^{l(\sigma)}\operatorname{tr}_{\mathcal{F}_{-\sqrt{p}(\sigma(\alpha+\hat{\lambda}+\rho)-\rho)+\bar{\lambda}}}(q^{L_0-\frac{c}{24}})\\
&=\sum_{\sigma\in W}(-1)^{l(\sigma)}\operatorname{tr}_{\mathcal{F}_{-\sqrt{p}(\alpha+\hat{\lambda}-\epsilon_{\lambda}(\sigma))+\overline{\sigma\ast\lambda}}}(q^{L_0-\frac{c}{24}}).\nonumber
\end{align}
The author expects that $\mathcal{W}_{-\sqrt{p}\alpha+\lambda}$ has a certain resolution as \cite[(3.14)]{ArF}.
\end{rmk}

Finally, we give a conjecture on narrow screening operators.
For $\lambda\in\Lambda$, $1\leq i_1,\dots,i_n\leq l$, and $\sigma\in W$ that has a minimal expression $\sigma=\sigma_{i_n}\cdots\sigma_{i_1}$, let $F_{\sigma,\lambda}$ denote the composition of narrow screening operators $F_{i_n,\sigma_{i_n}\cdots\sigma_{i_1}\ast\lambda}\cdots F_{i_1,\lambda}\in\operatorname{Hom}_{\mathbb{C}}(V_{\sqrt{p}Q+\lambda},V_{\sqrt{p}Q+\sigma_{i_n}\cdots\sigma_{i_1}\ast\lambda})$.
\begin{conj}\label{conjnarrow}
For any $\lambda\in\Lambda$ and $\sigma\in W$, $F_{\sigma,\lambda}$ is independent of the
choice of a
minimal expression of $\sigma\in W$.
Moreover, we have
\begin{align}\label{alignnarrow}
\Phi_{\lambda}(H^0(\xi_\lambda))=\bigcap\operatorname{Im}F_{\tau,\tau^{-1}\ast\lambda},
\end{align}
where the intersection runs over all
$\tau\in W$ such that $F_{\tau,\tau^{-1}\ast\lambda}\not=0$.
\end{conj}

\appendix
\section{A proof of Lemma \ref{condequiv}}\label{sectioncond}
For $1\leq i_1,\cdots,i_n\leq l$, denote $\sigma_{i_n\cdots i_1}$ by $\sigma_{i_n}\cdots\sigma_{i_1}$. We use the letter $\sigma_{i_0}$ for the identity map $\operatorname{id}_W$ of $W$ throughout this subsection.
\begin{lemm}\label{condequiv2}
The condition Lemma \ref{condequiv} \eqref{condition1} is equivalent to the following condition:
for $1\leq n\leq l(w_0)$, we have
\begin{align}\label{(1502)}
\epsilon_\lambda(\sigma_{i_{n}\cdots i_0})=\sum_{j=0}^{n-1}\epsilon_{\sigma_{i_{j}\cdots i_0}\ast\lambda}(\sigma_{i_{j+1}}).
\end{align} 
\end{lemm}
\begin{proof}
Let us write $\sigma$ for $\sigma_{i_{n}\cdots i_1}$ and $i$ for $i_{n+1}$, respectively. Then we have
\begin{align}
\epsilon_{\lambda}(\sigma)&=\sigma_i\epsilon_{\lambda}(\sigma)+(\epsilon_{\lambda}(\sigma),\alpha_i)\alpha_i\\
&=\epsilon_\lambda(\sigma_i\sigma)+\epsilon_{\sigma_i\sigma\ast\lambda}(\sigma_i)+\alpha_i+\delta_{(\sqrt{p}\overline{\sigma\ast\lambda},\alpha_i),p-1}\alpha_i+(\epsilon_{\lambda}(\sigma),\alpha_i)\alpha_i\label{bunki}
\end{align}
where the second equality follows from \eqref{(401)}. If $(\sqrt{p}\overline{\sigma\ast\lambda},\alpha_i)=p-1$, then by \eqref{401+1}, \eqref{bunki} is equal to
\begin{align}
\epsilon_\lambda(\sigma_i\sigma)+\alpha_i+(\epsilon_{\lambda}(\sigma),\alpha_i)\alpha_i=\epsilon_\lambda(\sigma_i\sigma)-\epsilon_{\sigma\ast\lambda}(\sigma_i)+(\epsilon_{\lambda}(\sigma),\alpha_i)\alpha_i.
\end{align}
Thus, we obtain that 
\begin{align}\label{1502-0.9}
\epsilon_\lambda(\sigma_i\sigma)=\epsilon_\lambda(\sigma)+\epsilon_{\sigma\ast\lambda}(\sigma_i)-(\epsilon_{\lambda}(\sigma),\alpha_i)\alpha_i.
\end{align}
On the other hand, if $(\sqrt{p}\overline{\sigma\ast\lambda},\alpha_i)\leq p-2$, then by \eqref{401-1}, \eqref{bunki} is equal to
\begin{align}
\epsilon_\lambda(\sigma_i\sigma)-\epsilon_{\sigma\ast\lambda}(\sigma_i)+(\epsilon_{\lambda}(\sigma),\alpha_i)\alpha_i
\end{align}
Thus, we obtain that
\begin{align}\label{1502-1}
\epsilon_\lambda(\sigma_i\sigma)=\epsilon_\lambda(\sigma)+\epsilon_{\sigma\ast\lambda}(\sigma_i)-(\epsilon_{\lambda}(\sigma),\alpha_i)\alpha_i.
\end{align}
If $\lambda\in\Lambda$ satisfies the condition in Lemma \ref{condequiv} \eqref{condition1}, then by \eqref{1502-0.9} and \eqref{1502-1}, for any $0\leq n\leq l(w_0)-1$, we have
\begin{align}\label{recursion}
\epsilon_{\lambda}(\sigma_{i_{n+1}\cdots i_0})=\epsilon_\lambda(\sigma_{i_{n}\cdots i_0})+\epsilon_{\sigma_{i_{n}\cdots i_0}\ast\lambda}(\sigma_{i_{n+1}}).
\end{align}
 By solving the recurrence relation \eqref{recursion} for $n$, we obtain \eqref{(1502)} for $1\leq n\leq l(w_0)$.
The converse is clear from \eqref{1502-0.9} and \eqref{1502-1}.
\end{proof}
\begin{rmk}
By combining Lemma \ref{condequiv} \eqref{condition1} with Lemma \ref{condequiv2}, the condition Lemma \ref{condequiv} \eqref{condition1} and Lemma \ref{condequiv2} are equivalent to the following: For $1\leq n\leq l(w_0)-1$, we have
\begin{align}\label{(1502.5)}
(\sum_{j=0}^{n-1}\epsilon_{\sigma_{i_{j}\cdots i_0}\ast\lambda}(\sigma_{i_{j+1}}),\alpha_{i_{n+1}})=0.
\end{align} 
\end{rmk}

\begin{lemm}\label{indepofminexp}
The conditions in Lemma \ref{condequiv} and Lemma \ref{condequiv2} are independent of the choice of a minimal expression of the longest element $w_0\in W$.
\end{lemm}
\begin{proof}
By Lemma \ref{condequiv2}, it is enough to show that the condition Lemma \ref{condequiv} \eqref{condition1} is independent of the choice of a minimal expression $w_0$. Let $w_0=\sigma_{i_{l(w_0)}\cdots{i_1}}$ be a minimal expression of $w_0$ that satisfies the condition Lemma \ref{condequiv} \eqref{condition1}. For $\sigma=\sigma_{i_{n-1}\cdots{i_1}}$, we have
\begin{align}\label{app1}
(\epsilon_\lambda(\sigma),\alpha_{i_n})=0,~(\epsilon_\lambda(\sigma_{i_n}\sigma),\alpha_{i_{n+1}})=(\epsilon_\lambda(\sigma)+\epsilon_{\sigma\ast\lambda}(\sigma_{i_{n}}),\alpha_{i_{n+1}})=0.
\end{align}
When $(\alpha_{i_n},\alpha_{i_{n+1}})=0$, the second equation in \eqref{app1} leads $(\epsilon_\lambda(\sigma),\alpha_{i_{n+1}})=0$.
Thus, by \eqref{401-1} and \eqref{(401)}, we have
\begin{align}\label{omake1}
\epsilon_{\lambda}(\sigma_{i_{n+1}}\sigma)=\epsilon_\lambda(\sigma)+\epsilon_{\sigma\ast\lambda}(\sigma_{i_{n+1}}).
\end{align}
By \eqref{omake1} and the first equation in \eqref{app1}, we also obtain that
$(\epsilon_{\lambda}(\sigma_{i_{n+1}}\sigma),\alpha_{i_n})=(\epsilon_\lambda(\sigma)+\epsilon_{\sigma\ast\lambda}(\sigma_{i_{n+1}}),\alpha_{i_n})=0$. Therefore, we have
\begin{align}\label{app2}
(\epsilon_{\lambda}(\sigma),\alpha_{i_{n+1}})=(\epsilon_{\lambda}(\sigma_{i_{n+1}}\sigma),\alpha_{i_n})=0.
\end{align}

When $i_{n\pm 1}=j$ and $(\alpha_{i_n},\alpha_j)=-1$, for $\sigma=\sigma_{i_{n-2}}\cdots\sigma_{i_1}$,  we have
\begin{align}\label{want}
(\epsilon_\lambda(\sigma),\alpha_{j})=(\epsilon_\lambda(\sigma_j\sigma),\alpha_{i_n})
=(\epsilon_\lambda(\sigma_{i_n}\sigma_j\sigma),\alpha_j)=0.
\end{align}
It is straightforward to verify that \eqref{want} holds if and only if 
\begin{align}\label{want1}
\sqrt{p}((\overline{\sigma\ast\lambda},\alpha_{i_n}+\alpha_j))\leq p-2.
\end{align}
Since \eqref{want1} is symmetrical with respect to $i_n$ and $j$, we also have 
\begin{align}\label{want0}
(\epsilon_\lambda(\sigma),\alpha_{i_n})=(\epsilon_\lambda(\sigma_{i_n}\sigma),\alpha_{j})
=(\epsilon_\lambda(\sigma_{j}\sigma_{i_n}\sigma),\alpha_{i_n})=0.
\end{align}
Therefore, by \eqref{app2} and \eqref{want0}, the condition Lemma \ref{condequiv} \eqref{condition1} is preserved under relations in the Weyl group $W$.
\end{proof}

Before the proof of Lemma \ref{condequiv}, we set up some notations.
For a sequence ${\bf m}=(m_1,\dots,m_i)$ of elements in $\Pi=\{1,\cdots, l\}$, let $r_{{\bf m}}$ denote the corresponding element $\sigma_{m_1}\cdots\sigma_{m_i}$ in $W$.
Also, for sequences ${\bf m}=(m_1,\dots,m_i)$ and ${\bf n}=(n_1,\dots,n_j)$ of elements in $\Pi$, we write ${\bf m}{\bf n}$ for the sequence $(m_1,\dots,m_i,n_1,\dots,n_j)$.
Let us take the sequences ${\bf s}_k$ of elements in $\Pi$ as in Table \ref{table0}, and set
$(i_{l(w_0)},\dots,i_1)={\bf s}_1\cdots{\bf s}_l$.
Then we obtain the minimal expression
\begin{align}\label{minexpofw_0}
w_0=r_{{\bf s}_1\cdots{\bf s}_l}=\sigma_{i_{l(w_0)}\cdots i_1}
\end{align}
of $w_0$ given in \cite{BKOP}.
For convenience, set $\sigma_{i_0}=r_{{\bf s}_{l+1}}=\operatorname{id}_W$.
For $0\leq m<n\leq l(w_0)$, $\tau=\sigma_{i_m}\cdots\sigma_{i_0}$, $\tau'=\sigma_{i_n}\cdots\sigma_{i_0}$, let $(\epsilon_{\sigma\ast\lambda})_{\tau\leq_L\sigma\leq_L\tau'}$ be the sequence
\begin{align}
(\epsilon_{\sigma\ast\lambda})_{\tau\leq_L\sigma\leq_L\tau'}=(\epsilon_{\tau\ast\lambda}(\sigma_{i_{m+1}}),\epsilon_{\sigma_{i_{m+1}}\tau\ast\lambda}(\sigma_{i_{m+2}}),\cdots,\epsilon_{\sigma_{i_{n-1}\cdots i_{m+1}}\tau\ast\lambda}(\sigma_{i_{n}}))\nonumber
\end{align}
of elements in $P$.
For $1\leq i\leq l$, we set
\begin{align}
(\epsilon_{\sigma\ast\lambda})_{(i+1,i)}=(\epsilon_{\sigma\ast\lambda})_{r_{{\bf s}_{i+1}\cdots{\bf s}_{l+1}}\leq_L\sigma\leq_Lr_{{\bf s}_{i}\cdots{\bf s}_{l+1}}}.
\end{align}

\begin{proof}[Proof of Lemma \ref{condequiv}]

We show that the condition Lemma \ref{condequiv} \eqref{condition2} and that in Lemma \ref{condequiv2} are equivalent.
By Lemma \ref{indepofminexp}, it is enough to consider the minimal expression \eqref{minexpofw_0} of $w_0$.

It is straightforward to show the following: first, when $(\sqrt{p}\bar\lambda+\rho,\theta)<p$, the list of $(\epsilon_{\sigma\ast\lambda})_{\operatorname{id}\leq_L\sigma\leq_Lw_0}$ is given by Table \ref{table}.
Second, when $(\sqrt{p}\bar\lambda+\rho,\theta)=p$, the list of $(\epsilon_{\sigma\ast\lambda})_{\operatorname{id}\leq_L\sigma\leq_Lw_0}$ is given by changing Table \ref{table} as 
\begin{align}\label{othercase}
\begin{cases}
\epsilon_{\sigma_1r_{{\bf s}_l}\ast\lambda}(\sigma_1)=-2\omega_1+\omega_2,~\epsilon_{\sigma_2r_{{\bf s}_{l-1}}r_{{\bf s}_l}\ast\lambda}(\sigma_2)=\omega_1-\omega_2+\omega_3,&(A_l)\\
\epsilon_{\sigma_1r_{{\bf s}_l}\ast\lambda}(\sigma_1)=-2\omega_1+\omega_2,~\epsilon_{r_{{\bf s}_l}\ast\lambda}(\sigma_2)=\omega_1-\omega_2,&(D_l)\\
\epsilon_{\sigma_6r_{{\bf s}_6}\ast\lambda}(\sigma_6)=-2\omega_6+\omega_5,~\epsilon_{\sigma_{321}r_{{\bf s}_6}\ast\lambda}(\sigma_5)=\omega_3-\omega_5+\omega_6,&(E_6)\\
\epsilon_{\sigma_7r_{{\bf s}_7}\ast\lambda}(\sigma_7)=-2\omega_7+\omega_6,~\epsilon_{\sigma_{5321}r_{{\bf s}_7}\ast\lambda}(\sigma_6)=-\omega_6+\omega_5+\omega_7,&(E_7)\\
\epsilon_{r_{{\bf s}_7}\sigma_8\ast\lambda}(\sigma_8)=-2\omega_8+\omega_7,~\epsilon_{\sigma_8r_{{\bf s}_7}\sigma_8\ast\lambda}(\sigma_7)=-\omega_7+\omega_8,&(E_8)
\end{cases}
\end{align}
and others are the same as Table \ref{table}. Third, when $(\sqrt{p}\bar\lambda+\rho,\theta)\leq p$, by Table \ref{table} and \eqref{othercase}, the equation \eqref{(1502.5)} holds for any $1\leq n\leq l(w_0)$.
Finally, when $(\sqrt{p}\bar\lambda+\rho,\theta)>p$, we have 
$(\sum_{j=0}^{n-1}\epsilon_{\sigma_{i_{j}\cdots i_0}\ast\lambda}(\sigma_{i_{j+1}}),\alpha_{i_{n+1}})>0$
for some $1\leq n\leq l(w_0)-1$. Thus, Lemma \ref{condequiv} is proved.
\end{proof}

\begin{cor}\label{cortable}
For $\lambda\in\Lambda$ such that $(\sqrt{p}\bar\lambda+\rho,\theta)\leq p$, we have $\epsilon_\lambda(w_0)=-\rho$.
\end{cor}
\begin{proof}
By Lemma \ref{condequiv2}, for the minimal expression \eqref{minexpofw_0} of $w_0$, we have $\epsilon_{\lambda}(w_0)=\sum_{j=1}^{l(w_0)}\epsilon_{i_{j-1}\cdots i_0\ast\lambda}(\sigma_{i_j})$.
Then by Table \ref{table} and \eqref{othercase}, the claim is proved.
\end{proof}

\begin{table}
\begin{tabular}{|l|l|}
\hline
$\mathfrak{g}$&\text{${\bf s}_i$ for $1\leq i\leq l$}\\\hline
$(A_l)$&${
\begin{aligned}
&{\bf s}_i=(l+1-i,\dots,l)\quad(1\leq i\leq l).\\
\end{aligned}
}$\\\hline
$(D_l)$&${
\begin{aligned}
&{\bf s}_i=(l+1-i,\dots,l-2,l,l-1,l-2,\dots,l+1-i)
\quad (3\leq i\leq l),\\
&{\bf s}_2=(l-1),~{\bf s}_1=(l).\nonumber
\end{aligned}
}$\\\hline
$(E_6)$&${
\begin{aligned}
&\text{${\bf s}_i$ is the same as for type $D_5$ for $1\leq i\leq 5$.}\nonumber\\
&{\bf s}_6=(6,5,3,4,2,1,3,2,5,3,4,6,5,3,2,1).\nonumber
\end{aligned}
}$\\\hline
$(E_7)$&${
\begin{aligned}
&\text{${\bf s}_i$ is the same as for type $E_6$ for $1\leq i\leq 6$.}\nonumber\\
&{\bf s}_7=(7,6,5,3,4,2,1,3,2,5,3,4,6,5,3,2,1,7,6,5,3,4,2, 3,5,6,7).
\end{aligned}
}$\\\hline
$(E_8)$&${
\begin{aligned}
&\text{${\bf s}_i$ is the same as for type $E_7$ for $1\leq i\leq 7$.}\nonumber\\
&{\bf s}_8=(8){\bf s}_7(8){\bf s}_7(8).
\end{aligned}
}$\\ \hline
\end{tabular}
\caption{The list of ${\bf s}_i$ for $1\leq i\leq l$ (\cite[Table 1]{BKOP}).}
\label{table0}
\end{table}

\begin{table}
\begin{tabular}{|l|l|}
\hline
$\mathfrak{g}$&\text{$(\epsilon_{\sigma\ast\lambda})_{(i+1,i)}$ for $1\leq i\leq l$}\\\hline
$(A_l)$&${
\begin{aligned}
&(\epsilon_{\sigma\ast\lambda})_{(i+1,i)}=(-\omega_l,-\omega_{l-1}+\omega_l,\cdots,-\omega_k+\omega_{k+1})\quad(1\leq i\leq l).\\
\end{aligned}
}$\\\hline
$(D_l)$&${
\begin{aligned}
&(\epsilon_{\sigma\ast\lambda})_{(i+1,1)}=(-\omega_{l+1-i},-\omega_{l+2-i}+\omega_{l+1-i},\cdots, \\
&-\omega_{l-1}+\omega_{l-2},-\omega_l,-\omega_{l-2}+\omega_{l-1}+\omega_l, -\omega_{l-3}+\omega_{l-2},\cdots,\nonumber\\
&-\omega_{l+2-i}+\omega_{l+3-i},-\omega_{l+1-i}+\omega_{l+2-i})\quad (3\leq i\leq l),\\
&(\epsilon_{\sigma\ast\lambda})_{(3,2)}=(-\omega_{l-1}),(\epsilon_{\sigma\ast\lambda})_{(2,1)}=(-\omega_{l}).\nonumber
\end{aligned}
}$\\\hline
$(E_6)$&${
\begin{aligned}
&\text{$(\epsilon_{\sigma\ast\lambda})_{(i+1,i)}$ is the same as for type $D_5$ for $1\leq i\leq 5$.}\nonumber\\
&(\epsilon_{\sigma\ast\lambda})_{(7,6)}=(-\omega_1,-\omega_2+\omega_1,-\omega_3+\omega_2,-\omega_5+\omega_3,-\omega_6+\omega_5,\\
&-\omega_4,-\omega_3+\omega_4,-\omega_5+\omega_3+\omega_6,-\omega_2,-\omega_3+\omega_2+\omega_5,-\omega_1,\nonumber\\ 
&-\omega_2+\omega_1+\omega_3,-\omega_4,-\omega_3+\omega_2+\omega_4,-\omega_5+\omega_3,-\omega_6+\omega_5).\nonumber
\end{aligned}
}$\\\hline
$(E_7)$&${
\begin{aligned}
&\text{$(\epsilon_{\sigma\ast\lambda})_{(i+1,i)}$ is the same as for type $E_6$ for $1\leq i\leq 6$.}\nonumber\\
&(\epsilon_{\sigma\ast\lambda})_{(8,7)}=(-\omega_7,-\omega_6+\omega_7,-\omega_5+\omega_6,-\omega_3+\omega_5,-\omega_2+\omega_3,\\
&-\omega_4,-\omega_3+\omega_2+\omega_4,-\omega_5+\omega_3,-\omega_6+\omega_5,-\omega_7+\omega_6,-\omega_1,\nonumber\\
&-\omega_2+\omega_1,-\omega_3+\omega_2,-\omega_5+\omega_3,-\omega_6+\omega_5+\omega_7,-\omega_4,-\omega_3+\omega_4,\nonumber\\
&-\omega_5+\omega_3+\omega_6,-\omega_2,-\omega_3+\omega_2+\omega_5,-\omega_1,-\omega_2+\omega_1+\omega_3,-\omega_4,\nonumber\\
&-\omega_3+\omega_2+\omega_4,-\omega_5+\omega_3,-\omega_6+\omega_5,-\omega_7+\omega_6).\\
\end{aligned}
}$\\\hline
$(E_8)$&${
\begin{aligned}
&\text{$(\epsilon_{\sigma\ast\lambda})_{(i+1,i)}$ is the same as for type $E_7$ for $1\leq i\leq 7$.}\\
&(\epsilon_{\sigma\ast\lambda})_{(9,8)}=
(-\omega_8,-\omega_7+\omega_8,-\omega_6+\omega_7,-\omega_5+\omega_6,-\omega_3+\omega_5,\\
&-\omega_2+\omega_3,-\omega_4,-\omega_3+\omega_2+\omega_4,-\omega_5+\omega_3,-\omega_6+\omega_5,-\omega_7+\omega_6,\nonumber\\
&-\omega_1,-\omega_2+\omega_1,-\omega_3+\omega_2,-\omega_5+\omega_3,-\omega_6+\omega_5+\omega_7,-\omega_4,\nonumber\\
&-\omega_3+\omega_4,-\omega_5+\omega_3+\omega_6,-\omega_2,-\omega_3+\omega_2+\omega_5,-\omega_1,-\omega_2+\omega_1+\omega_3,\nonumber\\
&-\omega_4,-\omega_3+\omega_2+\omega_4,-\omega_5+\omega_3,-\omega_6+\omega_5,-\omega_7+\omega_6,-\omega_8+\omega_7,-\omega_7,\nonumber\\
&-\omega_6+\omega_7,-\omega_5+\omega_6,-\omega_3+\omega_5,-\omega_2+\omega_3,-\omega_4,-\omega_3+\omega_2+\omega_4,\nonumber\\
&-\omega_5+\omega_3,-\omega_6+\omega_5,-\omega_7+\omega_6+\omega_8,-\omega_1,-\omega_2+\omega_1,-\omega_3+\omega_2,\nonumber\\
&-\omega_5+\omega_3,-\omega_6+\omega_5+\omega_7,-\omega_4,-\omega_3+\omega_4,-\omega_5+\omega_3+\omega_6-\omega_2,\nonumber\\
&-\omega_3+\omega_2+\omega_5,-\omega_1,-\omega_2+\omega_1+\omega_3,-\omega_4,-\omega_3+\omega_2+\omega_4,\nonumber\\
&-\omega_5+\omega_3,-\omega_6+\omega_5,-\omega_7+\omega_6,-\omega_8+\omega_7).
\end{aligned}
}$\\ \hline
\end{tabular}
\caption{The list of $(\epsilon_{\sigma\ast\lambda})_{\operatorname{id}\leq_L\sigma\leq_Lw_0}$ when $(\sqrt{p}\bar\lambda+\rho,\theta)<p$.}
\label{table}
\end{table}

\end{document}